\documentclass[12pt,a4paper]{article}
\usepackage[utf8]{inputenc}
\usepackage{amsmath}
\usepackage{amsfonts}
\usepackage{amssymb}
\usepackage{amsthm}
\usepackage[colorlinks=true,linkcolor=blue]{hyperref}
\hypersetup{
	citecolor={blue}
}
\usepackage[left=2cm,right=2cm,top=2cm,bottom=2cm]{geometry}
\usepackage{xcolor}
\usepackage[all,cmtip]{xy}
\usepackage{tikz}
\usetikzlibrary{decorations.markings}
\usepackage{enumerate}
\usepackage{enumitem}
\setlist[enumerate]{label = (\roman*)}
\graphicspath{{bilder/}}

\newcommand{\U}{\mathrm{U}}
\newcommand{\SU}{\mathrm{SU}}

\newcommand{\Hom}{\mathrm{Hom}}

\newcommand{\mft}{\mathfrak{t}}

\newcommand{\RR}{\mathbb{R}}
\newcommand{\CC}{\mathbb{C}}
\newcommand{\ZZ}{\mathbb{Z}}
\newcommand{\QQ}{\mathbb{Q}}
\newcommand{\PP}{\mathbb{P}}
\newcommand{\NN}{\mathbb{N}}

\newtheorem{thm}{Theorem}[section]
\newtheorem{prop}[thm]{Proposition}

\theoremstyle{definition}
\newtheorem{rem}[thm]{Remark}
\newtheorem{defn}[thm]{Definition}
\newtheorem{ex}[thm]{Example}

\begin{document}

\title{Low-dimensional GKM theory}
\author{Oliver Goertsches\footnote{Philipps-Universit\"at Marburg, email:
		goertsch@mathematik.uni-marburg.de}, Panagiotis Konstantis\footnote{Philipps-Universit\"at Marburg, email: pako@mathematik.uni-marburg.de}, and Leopold
	Zoller\footnote{Ludwig-Maximilians-Universit\"at M\"unchen, email: leopold.zoller@mathematik.uni-muenchen.de}}

\maketitle

\begin{abstract}
	GKM theory is a powerful tool in equivariant topology and geometry that can be used to generalize classical ideas from (quasi)toric manifolds to more general torus actions. After an introduction to the topic this survey focuses on recent results in low dimensions, where the interaction between geometry and combinatorics turns out to be particularly fruitful.
\end{abstract}

{
\tableofcontents
}

\section{Introduction}

The seminal work of Thomas Delzant \cite{Delzant} is a
milestone in the field of torus actions, showing the wonderful interplay between torus actions,
geometric structures and combinatorics. The so-called \emph{Delzant correspondence} comprises two statements:  \emph{rigidity}, stating that a symplectic manifold $M^{2n}$ of
dimension $2n$ on which a torus of rank $n$ is acting in a Hamiltonian fashion -- i.e., a toric symplectic manifold -- is completely determined, up to equivariant symplectomorphism, by the image of its moment
map, which is a so called \emph{Delzant polytope}. And secondly, \emph{realization}: given the combinatorial datum of a Delzant polytope, one can construct a 
toric symplectic manifold whose momentum image is the given Delzant polytope.
Similar correspondences have also been observed for other types of torus actions, for example quasitoric \cite{DavisJ} or torus manifolds \cite{MasudaUnitary,  MasudaPanov}.

What these results have in common is that are naturally located in the realm of complexity zero actions, where the
\emph{complexity} of a torus action is defined to be the difference $\tfrac12\dim M -
\mathrm{rk}\,T$. If one wishes to generalize such correspondences between certain types of torus actions and combinatorial objects, there are several natural ways to proceed. One would be to increase the complexity gradually, i.e., to consider now actions of complexity one. Recently this has become an active field of research, both from the Hamiltonian and the topological point of view. Various properties that are known in the (quasi)toric case, such as the orbit space of the action or its equivariant cohomology, were investigated, see for instance \cite{KarshonTolmanQuotients, AyzenbergCherepanov, AyzenbergMasuda, HolmKessler}, and there also begin to emerge classification results for certain subclasses of actions, such as that of tall symplectic complexity one actions \cite{KarshonTolman} or certain topological complexity one actions in general position \cite{Ayzenberg}.

Another approach to obtain results in higher complexity is   given by \emph{GKM theory}, the topic of this survey. This theory and its main players, GKM actions or GKM manifolds (see Definition \ref{def:gkm manifold}), are named after Goresky, Kottwitz and MacPherson \cite{GKM}. 
 The
celebrated Chang-Skjelbred lemma (Theorem \ref{thm:csl over Q}) makes it possible, under the assumption of equivariant formality
(Definition \ref{def:equivariant formal}), to compute the equivariant cohomology
of a torus action entirely in terms of the 1-skeleton (Remark
\ref{rem:one_skeleton}) of the action. In GKM theory one relaxes the assumption on complexity completely, but instead imposes a condition
on the set of zero- and one-dimensional orbits to be as simple as possible. More precisely, the fixed point set is assumed to be finite, and the one-skeleton has to be a union of $T$-invariant
$2$-spheres, as is the case in the (quasi)toric setting. The orbit space of the $1$-skeleton then  naturally leads to a combinatorial object, specifically a graph,
where the vertices correspond to the fixed points and the edges to the invariant
$2$-spheres. Together with a natural labelling of the graph by weights of the isotropy representations at the fixed points, we arrive at the \emph{GKM graph} of the action, which acts as a replacement for the polytopes one associates to actions in the (quasi)toric world. 

This brings us back to the Delzant
correspondence which we now may try to generalize to the GKM setting. Note that it is possible
to define the notion of an \emph{abstract GKM graph}, a purely combinatorial object (cf.\ Section \ref{sec:geometricstuff}). The most basic version of the \emph{GKM correspondence}
is the map
\[
	\{\text{GKM manifolds}\} \longrightarrow \{\text{abstract GKM graphs}\}
\]
sending a GKM manifold to its GKM graph, and one may ask about \emph{rigidity} and \emph{realization} statements, i.e., in how far this map is a bijection -- modulo  appropriate isomorphisms on both sides, see Section
\ref{sec:GKMcorrespondence}. There are several variants of this correspondence, depending on the chosen coefficient ring in cohomology, as well as on the presence of additional invariant geometric structures.

The purpose of this survey is twofold: in Section  \ref{sec:gkm equiv cohom} we wish to give a comprehensible introduction to GKM theory focusing both on rational and integral aspects of the theory -- see also \cite{Tymoczko}, or \cite[Section 11]{eqcohomsurvey} for other general introductions to GKM theory. In the later sections we discuss recent developments on the GKM rigidity and realization questions. 

We argue in Section  \ref{sec:dim4} that in dimension $\leq 4$, GKM manifolds are automatically torus manifolds, so that this theory adds nothing new to the picture in these dimensions. In dimension $8$ and higher, which are discussed in Section \ref{sec:dim8}, there are no available results on GKM realization beyond the classical results in complexity zero. GKM rigidity fails
(Theorem \ref{thm:rigidity kaputt}), which is reasonable since in dimension $8$ and higher manifolds cannot be distinguished by cohomology and characteristic classes.

The study of the GKM correspondence turns out to be the most fruitful in low
dimensions, especially in dimension 6, which is the topic of Section \ref{sec:dim6}. Note that the most natural case is to
consider an action of the two-torus $T^2$ on a simply-connected $6$-manifold, as for the three-torus we would land in the realm of complexity zero actions and in case of a circle action we cannot obtain a GKM action.
Therefore in dimension $6$ GKM theory naturally intersects with the theory of complexity one actions. It turns out that, at least in the setting without additional geometric structures, in this dimension we have a fairly complete picture.

Concerning the rigidity question in dimension $6$, the GKM graph of a simply-connected integer GKM manifold with certain assumptions on the stabilizers encodes both the non-equivariant diffeomorphism and the equivariant homeomorphism type. For the first statement, which is Theorem
\ref{thm:dim6diffeotype} below, one uses the result from \cite{Wall, Jupp, Zubr} that an oriented,
simply-connected, closed $6$-manifold is uniquely determined up to
diffeomorphism by cohomological data (see \cite{MR1365849} for a nice
overview on this topic). For the second statement see  Theorem \ref{thm:onetotonecorrespondence}. On the other hand
rigidity fails in presence of nontrivial discrete isotropies even in dimension
$6$ (Example \ref{ex:rigidity kaputt in dim 6}). As an application of these rigidity statements, it follows that Tolman's \cite{Tolman} and Woodward's \cite{Woodward} examples of Hamiltonian non-Kähler actions and the Eschenburg flag are non-equivariantly diffeomorphic and equivariantly homeomorphic, see Example \ref{ex:twe}. In particular, it follows that Tolman's and Woodward's examples carry a (non-invariant) Kähler structure.

With regard to the realization question in dimension $6$ we mention two results. Some $3$-valent GKM
graphs are graph-theoretically fibrations over $n$-gons, which suggests that
such a graph fibration could be geometrically realized by a projectivization of a
complex rank $2$ vector bundle over a (quasi)toric $4$-manifold or $S^4$ (see Section
\ref{sec:gkm fibrations}). This turns out to be true, see Theorem \ref{thm:gkm fibration}. The realization problem in dimension 6 culminated in a recent
project, where we showed that all abstract $3$-valent GKM graphs with the necessary
conditions can be realized as $6$-dimensional GKM manifolds (cf.\ Theorem
\ref{thm:we gotem all in dimension 6}). While in Theorem \ref{thm:gkm fibration} we investigated the realization question for fibrations also with additional invariant almost complex, symplectic, and Kähler structures, the general realization problem remains open in the setting with geometric structures as of now.\\

\noindent {\bf{Acknowledgements.}} This work
is part of a project funded by the Deutsche Forschungsgemeinschaft (DFG, German
Research Foundation) - 452427095.

\section{GKM theory and equivariant cohomology}\label{sec:gkm equiv cohom}

\subsection{Associating a graph to a torus action}\label{sec:associatinggraph}

In GKM theory one investigates a certain class of torus actions on smooth manifolds by means of an associated combinatorial object, its GKM graph. In this section we will explain the construction of this object.

Let $T=T^k = S^1 \times \cdots \times S^1$ denote a compact $k$-dimensional torus. Its Lie algebra $\mft$ contains the lattice $\mathbb{Z}_\mathfrak{t}=\ker \exp$. We also identify the Lie algebra $\mathrm{Lie(S^1)}$ of $S^1$ with $\mathbb{R}$ such that $\mathbb{Z}\subset \mathbb{R}$ is the kernel of the exponential map on $S^1$. Then for any $\lambda\in \Hom(T,S^1)$ we may consider $d\lambda$ as an element of the weight lattice $\mathbb{Z}_\mft^*=\Hom(\mathbb{Z}_\mathfrak{t},\mathbb{Z})\subset \mft^*$ inside the dual of the Lie algebra. In fact this establishes an isomorphism $\Hom(T, S^1)\cong \mathbb{Z}_\mft^*\cong \mathbb{Z}^k$ of Abelian groups, where for the last identification we use the dual basis of the standard basis of $\mathfrak{t}=(\mathrm{Lie}(S^1))^k=\mathbb{R}^k$. Under these identifications, which will be used frequently, the tuple $(l_1,\ldots,l_k)\in \mathbb{Z}^k$ corresponds to the homomorphism $T^k\rightarrow S^1$ with $(t_1,\ldots,t_k)\mapsto t_1^{l_1}\cdot\ldots\cdot t_k^{l_k}$.

\begin{rem}
	We recall that any representation of $T$ on a complex vector space $V$
	decomposes as $V=V_0\oplus \bigoplus_{\lambda} V_\lambda$, where $V_0$ is the
	subspace of fixed vectors and, for a character $\lambda:T\to S^1 \subset \mathbb{C}$,
	$V_\lambda=\{v\in V\mid t\cdot v = \lambda(t)v \textrm{ for all }t\in T\}$ is
	the \emph{weight space} of $\lambda$. The elements $0\neq \alpha:=d\lambda\in
		{\mathbb{Z}}_\mft^*$ such that $V_\lambda\neq 0$ are called the \emph{weights}
	of the representation, and we also write $V_{\alpha}$ instead of $V_\lambda$. In case we are dealing with a representation of $T$ on
	a real vector space $V$, we can also introduce a notion of weights, as
	elements in ${\mathbb{Z}}_\mft^*/{\pm 1}$: decomposing $V=V_0\oplus W$ as
	$T$-modules, we find an invariant complex structure on $W$, turning $W$ into a
	complex representation. The weights of the real representation are by
	definition $\pm \alpha$, where $\alpha$ runs over the weights of the complex
	representation $W$. The weight space of $\pm \alpha$ is by definition $V_{\pm
				\alpha} = V_\alpha\oplus V_{-\alpha}$. Note that $\pm \alpha$ corresponds to a character $\lambda\colon T^n\rightarrow S^1$ up to sign and that $\ker \lambda$ is independent of the sign choice. 
	The action of $\ker\lambda$ on $V_{\pm \alpha}$ is trivial.
\end{rem}

Let us consider a $T$-action on a compact connected smooth manifold $M$. We assume that the action has finitely many fixed points, i.e., that the fixed point set
\[
	M^T = \{p\in M\mid t p = p \textrm{ for all } t\in T\}
\]
is finite. For each $p\in M^T$, we may consider the isotropy representation of $T$ on $T_pM$. The finiteness of the fixed point set implies that there are no fixed vectors in $T_p M$ as otherwise the equivariant exponential map would yield a positive dimensional fixed manifold through $p$. As a consequence, $T_p M$ decomposes as a sum of irreducible $2$-dimensional real $T$-representations with nontrivial weights. Observe that this necessarily implies that $M$ is of even dimension in case $M^T$ is nonempty. As of now several of these weights might belong to the same weight space in the above sense.

However we imposes an additional assumption of the weights: to be precise we note that the notion of pairwise linear independence is meaningful for elements of a vector space that are well-defined only up to sign, and demand that for each $p$, the weights of the isotropy $T$-representation at $p$ are pairwise linearly independent in $\mathbb{Z}_\mft^*/{\pm 1}\subset \mft^*/{\pm 1}$. In particular, all weight spaces of $T_pM$ are $2$-dimensional.

\begin{rem} 
Recall that, given an action of a compact Lie group $K$ on a smooth manifold
$M$, the set of $K$-fixed points forms a submanifold of $M$ (possibly with
connected components of varying dimension). Thus, given a weight $\alpha\in {\mathbb
		Z}_\mft^*/{\pm 1}$ of the isotropy representation at $p$ with character
$\lambda$, we can consider the submanifold $M^{\ker \lambda}$ of $M$, and its
component $N$ containing $p$. Its tangent space at $p$ is $T_p N =
	(T_pM)_{\alpha}$, i.e., equal to the weight space of $\alpha$.
\end{rem}
Thus under the above assumptions, for any weight $\alpha$ of the isotropy representation at $p$, there is a $2$-dimensional $T$-invariant submanifold $N$ containing $p$, such that $T_pN = V_\alpha$. By the equality of Euler characteristics $\chi(N) = \chi(N^T)$ (\cite{Kobayashi}, see also \cite[Theorem 9.3]{eqcohomsurvey} for a proof using equivariant cohomology) it follows that $N$ is a closed surface with positive Euler characteristic, and hence diffeomorphic to $S^2$ or $\mathbb{R}P^2$. We can rule out the latter case by assuming that $M$ is orientable: then, as the $T$-action induces an orientation on the normal bundle of $N$, the submanifold $N$ is orientable as well. Now we have $N=S^2$ and there are exactly two $T$-fixed points in $N$ as $\chi(N)=2$. We have shown:
\begin{prop}\label{prop:Ntwospheres} Let $M^{2n}$ be a compact, connected, orientable smooth manifold, equipped with an action of a torus $T$ with finite fixed point set, such that for all $p\in M^T$ the weights of the irreducible factors of the isotropy representation at $p$ are pairwise linearly independent. Then for each of the $n$ weights $\alpha$ there is an invariant two-sphere $S_\alpha$ through $p$ such that $T_p(S_\alpha)$ is the weight space of $\alpha$.
\end{prop}
\begin{rem}
	Any $T$-invariant two-sphere $S$ in $M$ arises in the way described above. In fact, denoting by $H\subset T$ the subgroup of elements acting trivially on $S$, we obtain a character $\lambda:T\to T/H\cong S^1$. As $\chi(S)=2$, the sphere $S$ contains two $T$-fixed points $p$ and $q$; the corresponding isotropy representations admit the weight $\alpha= \pm d\lambda$, whose weight spaces are tangent to $S$. In particular, under the assumptions of the proposition there exist only finitely many $T$-invariant two-spheres in $M$.
\end{rem}
This allows us to construct a labelled graph. In the situation of the proposition we
\begin{enumerate}
	\item draw one vertex for each element in $M^T$,
	\item draw one edge for each $T$-invariant two-sphere in $M$, connecting the vertices corresponding to the two $T$-fixed points in $S$.
	\item label each edge with the corresponding weight of the isotropy representation, as an element of ${\mathbb{Z}}_\mft^*/{\pm 1}$.
\end{enumerate}

\begin{rem}\label{rem:one_skeleton}
	The \emph{one-skeleton} of a $T$-action on a smooth manifold $M$ is the union of at most one-dimensional orbits
	\[
		M_1:=\{p\in M\mid \dim Tp\leq 1\}.
	\]
	Any $T$-invariant two-sphere is contained in the one-skeleton $M_1$. The assumptions we made in Proposition \ref{prop:Ntwospheres} have the purpose to ensure that the part of $M_1$ which is connected to the fixed points can be encoded purely combinatorially in the labelled graph constructed above. Note in particular that the vital assumption of linear independence of the weights is the key of reducing the wild world of possible candidates for $N$ as above and opening the door to a purely combinatorial description. However in order to complete the definition of a GKM manifold, as we will in Section \ref{sec:GKMmanifolds}, there is one crucial condition missing which will ensure that $M_1$ does in fact see all (equivariant) cohomological information about $M$. This condition was named \emph{equivariant formality} in \cite{GKM} and is an equally strong and natural condition from the world of equivariant cohomology. We will introduce and discuss it in Section \ref{sec:equivariant cohomology} after a brief introduction to equivariant cohomology (cf. Definition	\ref{def:equivariant formal}). For now we just note two things: firstly, in the setup of Proposition \ref{prop:Ntwospheres} the condition of equivariant formality is equivalent to the vanishing of the odd dimensional (non-equivariant) cohomology of $M$. Hence equivariant cohomology is not needed to give the definition of a GKM manifold. Secondly, the additional condition of equivariant formality nicely completes our set of assumptions in two ways: it will not only tell us how the topology of $M$ is encoded in $M_1$ but also assure that $M_1$ is precisely the graph of two spheres constructed above, which in combination tells us that all the cohomological information is encoded in the labelled graph.
	
	As an addendum to the latter point, observe that as of now the one-skeleton of the action might be larger than the union of invariant two-spheres. 
For example, with the current set of conditions nobody prevents us to consider a free circle action, for which this associated graph would be empty, and the one-skeleton would be the whole manifold. Slightly more interesting, we could also consider the equivariant connected sum along a neighborhood of a regular orbit of an action with isolated fixed points and an action whose one-skeleton consists entirely of one-dimensional orbits. For example, the equivariant connected sum of the standard $T^2$-action on $S^4$ (see Example \ref{ex:s2n} below) with the product $S^3\times S^1$, where $T^2$ acts in standard fashion on $S^3\subset \CC^2$ and trivially on $S^1$, has a disconnected one-skeleton: one component consists of the two invariant two-spheres encoded in the graph of the action, and two further components are invariant two-tori without fixed points.
\end{rem}

\subsection{Examples} \label{sec:examples}

\begin{ex}\label{ex:s2n}
	Consider $M=S^{2n}$, the $2n$-dimensional sphere, as the unit sphere in $\CC^n\oplus \RR$. Then the $n$-dimensional torus $T= T^n=S^1\times \cdots \times S^1$ acts on $S^{2n}$ by rotating in the $n$ complex coordinates. The action has exactly two fixed points, $(0,\ldots,0,\pm 1)$, which are connected by the $n$ invariant two-spheres $S^{2n}\cap (0\oplus\cdots \oplus 0 \oplus \CC \oplus 0\oplus \cdots \oplus 0 \oplus \RR)$. The corresponding weights are, up to sign, the elements of the dual basis $\{e_i\}$ of the standard basis of the Lie algebra $\mft$.
	\begin{center}
		\begin{tikzpicture}

			\node (a) at (0,0)[circle,fill,inner sep=2pt] {};
			\node (b) at (6,0)[circle,fill,inner sep=2pt]{};
			\node at (3,.85) {\textcolor{blue}{$\pm e_2$}};
			\node at (3,1.45) {\textcolor{blue}{$\pm e_1$}};
			\node at (3,-.87) {\textcolor{blue}{$\pm e_{n-1}$}};
			\node at (3,-1.45) {\textcolor{blue}{$\pm e_n$}};

			\draw [very thick](a) to[in=160, out=20] (b);
			\draw [very thick](a) to[in=140, out=40] (b);
			\draw [very thick](a) to[in=200, out=-20] (b);
			\draw [very thick](a) to[in=220, out=-40] (b);
			\draw [very thick, dotted] (3,-.3) -- (3,.3);

		\end{tikzpicture}
	\end{center}
\end{ex}

\begin{ex}\label{ex:CPn}
	Consider $\CC P^n$, complex projective $n$-space, with the $T^n$-action given in homogeneous coordinates by
	\[
		(t_0,\ldots,t_{n-1})\cdot [z_0:\ldots :z_n] = [t_0z_0:\ldots:t_{n-1}z_{n-1}:z_n].
	\]
	This action has the $n+1$ fixed points $[1:0:\ldots :0],[0:1:0:\ldots:0],\ldots,[0:\ldots :0:1]$, and for any pair of fixed points there is an invariant $S^2 \cong \CC P^1$ of the form $\{[0:\ldots:0:z:0:\ldots:0:w:0:\ldots:0]\}$. The corresponding graph is thus the complete graph on $n+1$ vertices. 
	\begin{figure}[h]
	\centering
	\includegraphics{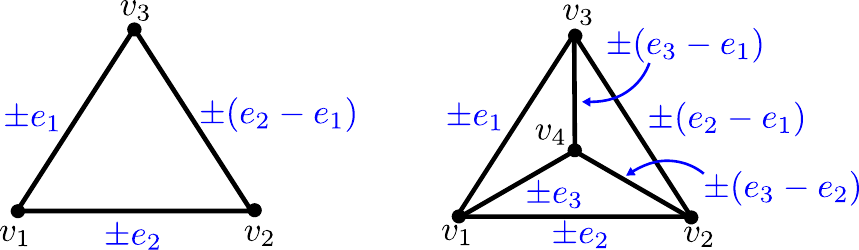}
	\caption{Graphs of $\CC P^2$ and $\CC P^3$}
	\label{fig:cp2cp3}
\end{figure}
\end{ex}

\begin{rem}\label{rem:Hamiltonian}
Recall that a $T$-action on a symplectic manifold $(M,\omega)$ is \emph{Hamiltonian} if $T$ leaves $\omega$ invariant and there is a \emph{momentum map} $\mu\colon M\rightarrow \mathfrak{t}^*$. The latter is a $T$-invariant map defined by the property that for any $\xi\in\mathfrak{t}$ the exterior differential of the function $p\mapsto \mu(p)(\xi)$ agrees with the contraction of $\omega$ along the fundamental vector field of $\xi$.  The momentum map $\mu$ maps every invariant $2$-sphere onto an interval in $\mathfrak{t}^*$ whose slope is given by the weight associated to the sphere -- this follows directly from the fact that by definition of a momentum map, for any $p\in M$ the image of $d\mu_p$ is the annihilator of ${\mathfrak{t}}_p$. In this setting, the image of the union of the $T$-invariant two-spheres gives an edgewise affine linear immersion of the graph in the vector space $\mft^*$. Note that the edges may meet in points different from the images of the fixed points. This kind of immersion gives a nice way to display certain labelled graphs (see the examples below). Reading off the exact labels of the GKM graph from the slopes of the embedded edges needs some piece of additional information e.g.\ such as the labels being primitive.
\end{rem}

\begin{ex} \label{ex:generaltoric}
	Generalizing Example \ref{ex:CPn}, we may consider any symplectic toric manifold $M$. As is well-known, the momentum image is a convex polytope, called the momentum polytope. Its vertices are exactly the images of the $T$-fixed points in $M$. The preimages of an $i$-dimensional face is a $2i$-dimensional invariant symplectic submanifold. In particular, the edges of the momentum polytope correspond to $T$-invariant symplectic two-spheres, hence the graph of $M$ is the one-skeleton of the momentum polytope. The label of an edge is the primitive element pointing in direction of its slope, as explained above.
\end{ex}

\begin{ex}\label{ex:flag}
	Guillemin--Holm--Zara considered, on homogeneous spaces of the form $G/H$, where $H\subset G$ are compact and connected Lie groups of equal rank, the action of a maximal torus $T\subset H$ by left multiplication. These actions always satisfy the conditions of the previous section \cite{GHZhom}, and they determined the associated graph in terms of the root systems of $H$ and $G$ with respect to the common maximal torus $T$. Briefly, the vertices are in one-to-one correspondence with the quotient of Weyl groups $W(G)/W(H)$ (both with respect to the maximal torus $T$), labels of the edges at the origin are in one-to-one correspondence to the roots of $G$ that are not roots of $H$, and at the other edges they are determined by $W(G)$-equivariance. As an example, the GKM graph of the $T^2$-action on the flag manifold $\SU(3)/T^2$ is as follows, where the labels of the edges are given by the primitive vectors pointing in direction of the corresponding slopes. The picture is strongly related to the momentum image of the $T^2$-action, when considering the flag manifold as a coadjoint orbit. See the explanation in Remark \ref{rem:Hamiltonian}.
	\begin{center}
		\begin{tikzpicture}
			\draw[step=1, dotted, gray] (-3.5,-4.5) grid (3.5,2.5);

			\draw[very thick] (1,1) -- ++(-3,0) -- ++(4,-4) -- ++(0,3);
			\draw[very thick] (2,0) -- ++(-1,1);
			\draw[very thick] (2,-3)--++(-1,0);
			\draw[very thick] (1,-3)--++(0,4);
			\draw[very thick] (1,-3) -- ++(-3,3) -- ++ (4,0);
			\draw[very thick] (-2,0) -- ++(0,1);

			\node at (1,1)[circle,fill,inner sep=2pt]{};

			\node at (-2,1)[circle,fill,inner sep=2pt]{};

			\node at (1,-3)[circle,fill,inner sep=2pt]{};

			\node at (-2,0)[circle,fill,inner sep=2pt]{};

			\node at (2,0)[circle,fill,inner sep=2pt]{};

			\node at (2,-3)[circle,fill,inner sep=2pt]{};

		\end{tikzpicture}
	\end{center}

\end{ex}
\begin{ex}\label{ex:eschenburg}
	Eschenburg's twisted flag manifold $SU(3)//T^2$ is maybe the most famous example of a biquotient. Concretely, it is given as the orbit space of the action of the two-dimensional torus $T^2$ on $\SU(3)$
	\[
		(s,t)\cdot A:=\left(\begin{matrix}s^2t^2 \\ & 1 \\ & & 1\end{matrix}\right) A\left(\begin{matrix}\bar{s}\\ & \bar{t} \\ & & \bar{s}\bar{t}\end{matrix}\right).
	\]
	The interest in biquotients originated in Riemannian geometry, as a source of
	new examples of Riemannian manifolds with positive sectional curvature
	\cite{EschenburgHabil}; in particular the Eschenburg flag admits such a
	metric. In \cite{GKZsympbiquot} we constructed symplectic structurs with
	compatible Hamiltonian actions on several infinite families of biquotients, by
	considering them as the total spaces of symplectic fibrations. The Eschenburg
	flag admits a GKM action of a two-dimensional torus, acting in the description
	above as left multiplication of the maximal torus of $\SU(3)$. Its GKM graph
	is as follows, where as in the previous examples the labels are the primitive
	vectors in direction of the slopes. As in the previous example, this presentation of the GKM graph comes from the momentum image.
	\begin{center}
		\begin{tikzpicture}
			\draw[step=1, dotted, gray] (-3.5,-4.5) grid (3.5,2.5);

			\draw[very thick] (1,1) -- ++(-3,0) -- ++(4,-4) -- ++(0,3) -- ++(-1,1);
			\draw[very thick] (2,-3)--++(-1,+2)--++(0,2)++(0,-2)--++(-1,1)--++(2,0)++(-2,0)--++(-2,1);
			\node at (1,1)[circle,fill,inner sep=2pt]{};

			\node at (-2,1)[circle,fill,inner sep=2pt]{};

			\node at (1,-1)[circle,fill,inner sep=2pt]{};

			\node at (0,0)[circle,fill,inner sep=2pt]{};

			\node at (2,0)[circle,fill,inner sep=2pt]{};

			\node at (2,-3)[circle,fill,inner sep=2pt]{};

		\end{tikzpicture}
	\end{center}

\end{ex}

\subsection{Equivariant cohomology}\label{sec:equivariant cohomology}

Throughout this section we assume $G$ to be a
compact Lie group and $X$ a $G$-space. The goal of this section is to define a cohomology
theory which encodes not only the topology of $X$ but also the action of $G$ on
$X$. One obvious idea would be to take the cohomology of the quotient space
$X/G$ as equivariant cohomology of the action. But even the simple example of the standard rotation of
$X=S^2$ by $G=S^1$ shows that the homotopy type of the
quotient space (a real line segment) does not carry any interesting information
about the action. On the other hand, if $G$ is acting freely on a smooth manifold
$X$, then the orbit space $X/G$ is a closed manifold and hence does contain interesting cohomological information. Hence $H^*(X/G)$ can be considered as a reasonable candidate for a cohomology theory in the case of a free action.

One motivation for definition of equivariant cohomology given below is to, in the light of the above discussion, replace $X$ by a homotopy equivalent space on which $G$ acts freely.
This is achieved by means of the universal principal $G$-bundle. Recall that for every Lie group, there
is principal $G$-bundle $EG$ which is up to equivariant homotopy equivalence
uniquely determined by the fact, that $EG$ is weakly contractible and $G$ acts freely
from the right on $EG$, see \cite{Husemoller}. Its quotient $BG: = EG/G$ is
called a \emph{classifying space} for $G$.

	We now replace $X$ with the homotopy equivalent $EG \times X$ on which $G$ acts via the free diagonal action
	$g \cdot (e,x) := (eg^{-1},gx)$. We denote by $X_G :=EG \times_G X$ its quotient, which is also
	known as the \emph{Borel construction} or as the \emph{homotopy quotient}. Hence
	$X_G$ is the associated $X$-bundle to the principal $G$-bundle $EG \to BG$.
	Therefore we obtain a fibration
	\[
		X \longrightarrow X_G \longrightarrow BG.
	\]
	which is called the \emph{Borel fibration}. In what follows we fix a choice of universal bundle $EG\rightarrow BG$. Due to the uniqueness up to equivariant homotopy equivalence, the following definition is independent of this choice up to natural isomorphism.

\begin{defn}\label{def:equivariant cohomology}
	Let $\Lambda$ be a ring with unit. We define the \emph{equivariant cohomology}
	of a $G$-space $X$ by
	\[
		H_G^\ast(X; \Lambda) := H^\ast (X_G;\Lambda) = H^\ast (EG \times_G X; \Lambda)
	\]
	where $H^\ast(-, \Lambda)$ means the singular cohomology of space with
	coefficients in the ring $\Lambda$.
\end{defn}

\begin{rem}\label{rem:borelVSdeRham}
	For $\Lambda =\RR$ there are other possibilities to define equivariant cohomology on smooth manifolds, namely the so-called Cartan and Weil models. For the Cartan model cf.\ \cite{MR4205963} or \cite{eqcohomsurvey}; see \cite{Meinrenken} for an overview on the three models.
\end{rem}

\begin{ex}\label{ex:equivariant cohomology of a point}
	A very important special case is the trivial action of $G$ on a one point space $*$. Then $*_G=(EG\times*)/G\cong EG/G$, hence $H^\ast_G(*; \Lambda) = H^\ast(BG;\Lambda)$. For this article the
	most important group $G$ is the torus $T^k$. Next, we would like explain
	the cohomology of $BT^k$.

	We start with $k=1$, thus $T^1 =S^1$. For all $n \in \NN$ the complex Hopf fibrations
	\[
		S^1 \longrightarrow S^{2n+1} \longrightarrow \CC\PP^{n}
	\]
	are $S^1$--principal bundles. Hence $S^1$ acts also free on $S^\infty =
		\bigcup_{n \in \NN} S^{2n+1}$ with orbit space $\CC\PP^{\infty} = \bigcup_{n \in
			\NN} \CC\PP^{n}$. Since $S^\infty$ is weakly contractible, we see that $ES^1 =
		S^\infty$ and thus $BS^1 = \CC\PP^\infty$. The cohomology 	$H^*(\CC\PP^{\infty};\Lambda)$ is a polynomial ring with coefficients in $\Lambda$ in one variable
	\[
		H^\ast( \CC\PP^{\infty}; \Lambda) = \Lambda[x]
	\]
	where $x \in H^2( \CC\PP^{\infty}; \Lambda) \cong \Lambda$ is a generator. A higher dimensional torus $T^k$ acts freely in the obvious way on the weakly contractible space $S^\infty
		\times \ldots \times S^\infty$ ($k$ times). Hence $ET^k = S^\infty \times \ldots \times S^\infty$
	and therefore $BT^k = \CC\PP^{\infty}  \times \ldots \times \CC\PP^{\infty}$. We infer
	that
	\[
		H^\ast(BT^k; \Lambda) = \Lambda[x_1, \ldots,x_k]
	\]
	is a polynomial ring in $k$ variables such that $x_i \in H^2(BT^k; \Lambda)$.
\end{ex}

{\begin{rem}\label{rem:characters and cohomology of BT}
	An important fact is the identification of $\Hom(T, S^1)$ (the group of
	homomorphisms between $T$ and $S^1$) with $H^2(BT; \ZZ)$. Every $\alpha \in
		\Hom(T,S^1)$ defines a map $BT \to BS^1$ since $\alpha$ defines an
	$S^1$-bundle $L_{\alpha}:=ET \times_\alpha S^1$ over $BT$. Furthermore we have
	$BS^1 = \CC\PP^{\infty}$ which is a $K(\ZZ,2)$, thus the map $BT \to BS^1$
	defines an element in $H^2(BT; \ZZ)$ which is the first Chern class
	of $L_{\alpha}$. It is easy to see that the mapping $\Hom(T, S^1)
		\to H^2(BT;\ZZ)$ is a group isomorphism.
\end{rem}}

\begin{ex}\label{ex:S1}
As a first application of the above remark we compute $H_T^*(S^1;\ZZ)$ where $T$ acts via $\alpha\in\Hom(T,S^1)$. The Borel construction $S^1_T$ is just the circle bundle $L_\alpha$ over $BT$ defined above with Chern class $\alpha\in H^*(BT;\mathbb{Z})$. Hence the cohomology is $H_T^*(S^1;\ZZ)=H^*(BT;\ZZ)/(\alpha)$.
\end{ex}

\begin{ex}
	\label{ex:equiv cohomology of a free action}
	Let us consider a free $G$-space $X$. It follows
	that $H^\ast_G(X;\Lambda)$ is isomorphic to $H^\ast(X/G;\Lambda)$.
	To see this, note that $X \to X/G$ is a
	$EG$-fiber bundle by the Borel construction
	\[
		EG \longrightarrow EG \times_G X \longrightarrow X/G
	\]
	(since $EG$ is a $G$-space). But $EG$ is a weakly contractible space and by
	the long exact homotopy sequence of the fibration above it follows that $X_G$
	and $X/G$ are weakly homotopy equivalent. The latter implies that their
	cohomology rings coincide for all coefficient rings $\Lambda$.
\end{ex}

\begin{rem}\label{rem:algebrastructure} Given a $G$-equivariant map $f\colon X\rightarrow Y$ we obtain an induced map $X_G\rightarrow Y_G$ by applying $f$ in the second coordinate of the Borel construction. In particular we obtain an induced map $H_G^*(Y;\Lambda)\rightarrow H_G(X;\Lambda)$
making equivariant cohomology a functor on the category of $G$-spaces. In particular $H_G^*(X;\Lambda)$ is naturally an algebra over 
$H^*_G(*;\Lambda)=H^*(BG;\Lambda)$ by considering the equivariant map $X\rightarrow *$.
\end{rem}
With the naturality of equivariant cohomology established, it is not hard to see that all standard features from ordinary cohomology like the Mayer-Vietoris sequence and homotopy invariance carry over to the equivariant setting. One just applies the non-equivariant versions to the Borel construction. As a demonstration we have the following important

\begin{ex}\label{ex:S2}
We compute the equivariat cohomology of the $T$-action on $S^2$ defined by rotating along $\alpha\in\Hom(T,S^1)$. We use the Mayer-Vietoris sequence of $S^2=U\cup V$ where $U$ is the complement of the equator and $V$ is the complement of the two fixed poles. This decomposition may not be the most direct one for this special case but it has the advantage of generalizing to arbitrary graphs of two-spheres which will be helpful in the GKM description of equivariant cohomology, Proposition \ref{prop:GKMdescription}.

We observe that $V$ deformation retracts equivariantly to the fixed point set $\{N,S\}$ hence $H_T^*(V;\ZZ)= H^*_T(\{N\};\ZZ)\oplus H^*_T(\{S\};\ZZ)=H^*(BT;\ZZ)\oplus H^*(BT;\ZZ)$. The neighbourhood $U$ deformation retracts equivariantly onto the equator $S^1$. Hence by Example \ref{ex:S1}, $H^*_T(V;\ZZ)=H^*_T(S^1;\ZZ)=H^*(BT;\ZZ)/(\alpha)$. The intersection $U\cap V$ deformation retracts to two copies of $S^1$. Now on the cohomological level $U\cap V\rightarrow V$ induces the projection map $H^*(BT;\ZZ)^2\rightarrow (H^*(BT;\ZZ)/(\alpha))^2$ where the factor of each fixed point projects to the factor of the circle close to it and maps trivially to the other. This is indeed just the standard projection as all maps are maps of $H^*(BT;\ZZ)$-algebras as explained in Remark \ref{rem:algebrastructure}. The inclusion $U\cap V\rightarrow U$ induces the diagonal map. Dividing by the diagonal $\Delta$ we identify $(H^*(BT;\ZZ)/(\alpha))^2/\Delta\cong H^*(BT;\ZZ)/(\alpha)$ where the map is given by taking the difference of the two components. Using this, the Mayer-Vietoris sequence simplifies to the short exact sequence
\[0\rightarrow H_T^*(S^2;\ZZ)\rightarrow H^*(BT;\ZZ)\oplus H^*(BT;\ZZ)\rightarrow H^*(BT;\ZZ)/(\alpha)\rightarrow 0.\]
Hence $H_T^2(S^2;\ZZ)$ embeds as the subalgebra $\{(f,g)\in H^*(BT;\ZZ)^2~|~ f\equiv g \mod \alpha\}$ of the equivariant cohomology of the fixed point set.
\end{ex}

Let us now restrict to the case $G=T=T^k$ of a $k$-dimensional torus.

\begin{defn}\label{def:equivariant formal}
	We say that the action of $T$ on $X$ is \emph{equivariantly formal (over $\Lambda$)} if
	$H_T^\ast(X;\Lambda)$ is a free $H^\ast(BT;\Lambda)$-module.
\end{defn}

There are many equivalent conditions for equivariant formality over a field, see e.g.\ \cite[Theorem 7.3]{eqcohomsurvey}. Over the integers the situation is more subtle. For the purpose of this survey the following observation will be sufficient. We note that the proof relies on the Borel localization theorem which is responsible for the additional topological assumption of $X$ being a finite CW-complex. These may be relaxed, see e.g.\ the discussion surrounding \cite[Theorem 3.2.6]{AlldayPuppe}.
\begin{prop}\label{prop:ef} Let $X$ be a finite CW-complex with a $T$-action with discrete fixed points and $\Lambda=\mathbb{Q},\mathbb{Z}$. Then the following conditions are equivalent:
\begin{enumerate}
\item $X$ is equivariantly formal.
\item $H^{odd}(X;\Lambda)=0$.
\end{enumerate}
In this case the restriction map $H^*_T(X;\Lambda)\rightarrow H^*(X;\Lambda)$ is surjective and its kernel is precisely the ideal $H^+(BT;\Lambda) \cdot H_T^*(X;\Lambda)$.
\end{prop}
\begin{proof}
The equivalence of $(i)$ and $(ii)$ was proved e.g.\ in \cite[Lemma 2.1]{MasudaPanov}. Note that the authors state the lemma for manifolds but the proof can be applied to the more general setting, as well as $\Lambda=\mathbb{Q}$. For the latter case see also \cite[Corollary 4.2.3]{AlldayPuppe} and combine it with Borel localization. The vanishing of $H^{odd}(X;\Lambda)$ implies that the Serre spectral sequence of the Borel fibration collapses at the $E_2$-page. This implies the surjectivity of $H^*_T(X;\Lambda)\rightarrow H^*(X;\Lambda)$ as well as the description of the kernel. This is standard for field coefficients; the integral case is discussed e.g.\ in \cite[Theorem 1.1]{FranzPuppeIntegral}.
\end{proof}

The above proposition tells us that in the equivariantly formal case, the equivariant cohomology encodes the non-equivariant cohomology. Hence it suffices to study the equivariant cohomology. The key to doing this is the \emph{Chang-Skjelbred lemma} which tells us that in the equivariantly formal case $H_T^*(X;\Lambda)$ is entirely encoded in the one-skeleton $X_1$ of the action
(see Remark \ref{rem:one_skeleton}). The original rational version is due to {\cite[Lemma 2.3]{ChangSkjelbred}}. The more recent integral version was proved in {\cite[Corollary 2.2]{FranzPuppe}}. In both cases the injectivity of the first map in the sequence below is a consequence of the classical Borel localization theorem (see e.g.\ \cite[Theorem 3.2.6]{AlldayPuppe}).

\begin{thm}\label{thm:csl over Q}
	Let $T$ act on a finite CW-complex $X$ such that the action is equivariantly formal over $\Lambda=\mathbb{Q},\mathbb{Z}$. In the case $\Lambda=\mathbb{Z}$ assume additionally that for every $p \not \in X_1$ the
	isotropy group $T_p$ is contained in a proper subtorus. Then we have an
	exact sequence
	\[
		0 \longrightarrow H_T^\ast(X;\Lambda) \longrightarrow H_T^\ast(X^T; \Lambda) \longrightarrow H^{\ast
				+1}_T(X_1,X^T;\Lambda),
	\]
	where the last map is the connecting homomorphism in the long exact sequence of
	the pair $(X_1,X^T)$.
\end{thm}

The above theorem tells us that the image of $H_T^\ast(X;\Lambda) \to H_T^\ast(X^T;\Lambda)$ is equal (as subalgebras) to
the image of $H_T^\ast(X_1;\Lambda) \to H_T^\ast(X^T;\Lambda)$. Thus $H_T^\ast(X;\Lambda)$
is determined by the action on the one-skeleton of $X$.

Another important feature of equivariant formality is the following result regarding the combinatorics of the orbit type stratification. Its proof, which we shall omit, is a consequence of Borel localization. We point the reader towards \cite[Corollary 3.6.19]{AlldayPuppe} or \cite[Corollary 9.9]{eqcohomsurvey} for a dimension counting argument.
\begin{prop}\label{prop:efhasfixedpoints}
Let $X$ be a finite CW-complex with a rationally equivariantly formal $T$-action. Let $H\subset T$ be a subtorus. Then every component of $M^H$ contains a fixed point.
\end{prop}

Finally we would like to discuss equivariant vector bundles and equivariant
characteristic classes.

\begin{defn}\label{def:equivariant bundle and classes}
	Let $ \pi \colon V \to X$ be a vector bundle (either over $\RR$ or $\CC$). We
	say $\pi$ is a \emph{$G$-equivariant vector bundle} if there is a left-action of
	$G$ on $V$ where $\pi$ is $G$-equivariant and such that for all $g \in G$
	and $x \in X$ the map $\pi^{-1}(\{x\})=:V_x \to V_{g \cdot x}$, $v \mapsto g
		\cdot v$ is a linear.
\end{defn}

Every $G$-equivariant vector bundle $\pi \colon E \to X$ defines through the
Borel construction a vector bundle over its homotopy quotient:
\[
	\pi_G \colon V_G = EG \times_G V \to EG \times_G X = X_G, \quad
	[e,v] \mapsto [e, \pi(v)].
\]
\begin{defn}
	\label{def:equivariant char class}
	If $c$ denotes a characteristic class in a cohomology group with coefficients
	$\Lambda$ then we define the \emph{equivariant characteristic class} $c^G(V)$ to
	be  the characteristic class $c(V_G)$ of the vector bundle $V_G \to X_G$, thus
	$c^G(V)$ is an element of $H^\ast(X_G;\Lambda) = H_G^\ast(X;\Lambda)$.
\end{defn}

\begin{ex}
	\label{ex:equivariant vb and classes}
	Let $G$ be a Lie group acting smoothly from the left on a smooth manifold $M$.
	For $g \in G$ denote by $L_g \colon M \to M$ the corresponding diffeomorphism
	induced by the left action of $g$. $G$ acts on $TM$ by the differential of $L_g$, i.e.
	$g \cdot v := D(L_g)(v)  \in TM_{g \cdot p}$ for $v \in TM_p$ ($p \in M$). This
	action makes $\pi \colon TM \to M$ to a $G$-equivariant vector bundle.
\end{ex}

\begin{ex}
	\label{ex:equivarian char classes of trivial vb}
	Consider a complex linear representation of $T = (S^1)^k$ on $\CC^n$. First recall the
	decomposition of $\CC^n$ into complex $1$-dimensional irreducible
	subrepresentations, i.e. $\CC^n = \oplus_{\alpha} V_\alpha$, where $\alpha \in
		\mathrm{Hom}(T,S^1) \cong \ZZ^k$. We consider this representation as a
	$T$-equivariant complex vector bundle $V \to \ast$ over a point. The associated
	vector bundle of the corresponding homotopy quotients is given by
	\[
		ET \times_{T} \CC^n \longrightarrow ET \times_{T} \ast = BT =
		(\CC\PP^{\infty})^k
	\]
	Furthermore we have $ET \times_{T} \CC^n = \oplus_{\alpha} L_\alpha$
	where $L_\alpha = ET \times_{T} V_{\alpha}$. Since $c_1(L_{\alpha}) =
		\alpha$, where we identify $\mathrm{Hom}(T,S^1)$ as $H^2(BT;\mathbb{Z})$ as
	in Remark \ref{rem:characters and cohomology of BT}, we obtain for the total
	equivariant Chern class of $V \to \ast$
	\[
		c^T(V) = \prod_{\alpha}(1+\alpha) \in H_T^\ast(\ast) = H^\ast(BT;\mathbb{Z}).
	\]

\end{ex}

\subsection{GKM manifolds} \label{sec:GKMmanifolds}

\begin{defn}\label{def:gkm manifold}
	An effective action of a torus $T$ on a compact, connected, orientable smooth manifold $M$ is called an \emph{integer GKM action} (resp.\ a \emph{rational GKM action}), if
	\begin{enumerate}
		\item it is equivariantly formal over $\mathbb{Z}$ (resp.\ over $\mathbb{Q}$),
		\item the fixed point set $M^T$ is finite, and
		\item for all $p\in M^T$ the weights of the irreducible factors of the isotropy representation at $p$ are pairwise linearly independent.
	\end{enumerate}
	A manifold with such an action will also be called an \emph{integer (resp. rational) GKM manifold}.
\end{defn}
Recall from Proposition \ref{prop:ef} that the condition of equivariant
formality can be replaced by the vanishing of the odd-degree integral (resp.\ rational) cohomology.

\begin{defn}
	Given a GKM action, its associated labelled graph as constructed in Section \ref{sec:associatinggraph} is called the \emph{GKM graph of the action}.
\end{defn}

We complete the discussion from Section \ref{sec:associatinggraph} and in particular Remark \ref{rem:one_skeleton}. First of all, we note that for a GKM manifold the one-skeleton $M_1$ is exactly the space encoded by the GKM graph. To see this let $x\in M_1$. Then $x\in M^H$ for some codimension $1$ subtorus of $H$ and by Proposition \ref{prop:efhasfixedpoints} it follows that the component $N$ of $x$ in $M^H$ contains a fixed point $p$. But then $T_p N$ is a $T$-subrepresentation of $T_p M$ such that all stabilizers in $T_p N$ are of codimension $\leq 1$. It follows that $T_p N$ has to be one of the irreducible subrepresentations and hence $N$ is either an isolated fixed point or one of the invariant two-spheres.
 
Now equivariant formality implies that $H_T^*(M;\Lambda)$ is embedded as a $H^*(BT;\Lambda)$-subalgebra of \[H_T^*(M^T;\Lambda)=\bigoplus_{p\in M^T} H^*(BT;\Lambda)\] and by Theorem \ref{thm:csl over Q} the image is described completely by $M_1$ and hence by the GKM graph (where over $\mathbb{Z}$ we have to impose mild additional assumptions on the stabilizers). This is what is called the \emph{GKM description of equivariant cohomology}. Proposition \ref{prop:ef} tells us how to reconstruct the non-equivariant cohomology from this description.

While this covers the theory of how and why the GKM description works, the (equivariant) cohomology algebras and characteristic classes can be computed via nice combinatorial formulas in terms of the weights as described by the following proposition. To make sense of these descriptions recall how we have identified $\mathbb{Z}_t^*\cong \Hom (T,S^1)\cong H^2(BT;\mathbb{Z})$ as in Remark \ref{rem:characters and cohomology of BT}.

\begin{prop}\label{prop:GKMdescription}
Let $M$ be a GKM $T$-manifold over the coefficient ring $\Lambda=\mathbb{Z},\mathbb{Q}$. In case $\Lambda = \ZZ$ assume additionally that for all $p\notin M_1$, the isotropy group $T_p$ is contained in a proper subtorus of $T$. Then the inclusion $M^T\rightarrow M$ induces an isomorphism of $H^*_T(M;\Lambda)$ onto the subalgebra
 \[
			\left\{\left.(f(p))_p\in \bigoplus_{p\in M^T} H^*(BT;\Lambda) \,\,\,\right|  \,\, \begin{matrix} f(p) - f(q) \equiv 0 \text{ mod }\alpha \quad\text{if $p$ and $q$ lie in}\\ \text{a common invariant $2$-sphere with weight $\alpha$}  \end{matrix} \right\}
			      \]
\end{prop}
\begin{proof}
We prove the statement for $\Lambda=\mathbb{Z}$, the other case is analogous. By Theorem \ref{thm:csl over Q} it remains only to show that the above subalgebra is the image of the map $H_T^*(M_1;\Lambda)\rightarrow H_T^*(M^T;\Lambda)$ induced by the inclusion. To see this we consider the equivariant Mayer-Vietoris sequence of the graph $M_1$ obtained by covering it with $U=M_1\backslash M^T$ and $V$ a collection of small equivariant neighbourhoods around $M^T$ which deformation retract onto $M^T$. Then in the same way as in Example \ref{ex:S2} the Mayer-Vietoris sequence simplifies to
\[0\rightarrow H_T^*(M;\mathbb{Z})\rightarrow \bigoplus_{p\in M^T} H^*(BT;\mathbb{Z})\rightarrow \bigoplus_{S\subset M_1} H(BT;\mathbb{Z})/(\alpha(S))\]
where the map to the component $H^*(BT;\mathbb{Z})/(\alpha(S))$ component of some invariant $2$-sphere $S$ is given by forming the difference of the $H^*(BT;\mathbb{Z})$ components in the middle term which belong to the two fixed points of $S$. Hence the kernel of the right hand map is precisely the subalgebra described in the proposition.
\end{proof}

The description of the equivariant cohomology of a GKM action in terms of its
GKM graph as in Proposition \ref{prop:GKMdescription} suggests the following definition:
\begin{defn}[\cite{GuilleminZaraII}]
	For an abstract graph $\Gamma$ with edge set $E(\Gamma)$ and a function $\alpha\colon E(\Gamma)\rightarrow H^2(BT;\Lambda)$ the \emph{equivariant graph cohomology} is defined as
	\[
		H^*_T(\Gamma,\alpha;\Lambda) = \left\{\left.(f(u))_u\in \bigoplus_{u\in V(\Gamma)} H^*(BT;\Lambda) \,\,\,\right|  \,\, \begin{matrix} f(v) - f(w) \equiv 0 \text{ mod }\alpha(e) \\ \text{ for every edge $e$ between $v$ and $w$} \end{matrix} \right\}
	\]
	and its \emph{graph cohomology} as
	\[
		H^*(\Gamma,\alpha;\Lambda):=H^*_T(\Gamma,\alpha;\Lambda)/(H^{>0}(BT;\Lambda)\cdot H^*_T(\Gamma,\alpha;\Lambda)).
	\]
\end{defn}

\begin{rem}
In the situation of Proposition \ref{prop:GKMdescription} for the GKM graph $(\Gamma,\alpha)$ of $M$ one thus has
\begin{equation}\label{eq:cohomiso}
	H^*_T(M;\Lambda) \cong H^*_T(\Gamma,\alpha;\Lambda)\qquad \textrm{and}\qquad H^*(M;\Lambda) \cong H^*(\Gamma,\alpha;\Lambda)
\end{equation}
as $H^*(BT;\Lambda)$-, respectively $\Lambda$-algebras, where for the second isomorphism we use the last part of Proposition \ref{prop:ef}.
\end{rem}

Another nice feature of the GKM description is that one obtains a rather direct description of the (equivariant) characteristic classes. We sum this up in the following two propositions.

\begin{prop}
	\label{prop:equivariant chern classes for gkm}
	Let $T$ act on a compact, connected, orientable smooth
	manifold $M$ with finite fixed point set $M^T = \{p_1,\ldots,p_k\}$ and
	let $V \to M$ be a complex $T$-equivariant vector bundle.

	For
	$i=1,\ldots,k$, let $\alpha_{ij}$, $j=1,\ldots,m$, be the weights of the
	isotropy representation of $T$ at $V_{p_i}$, listed with multiplicities.
	Under the map
	\[
		H^*_T(M;\mathbb{Z})\longrightarrow H^*_T(M^T;\ZZ)
	\]
	which is induced by the inclusion, the equivariant total Chern class is given
	by
	\[
		\sum_{i=1}^{k} \prod_{j=1}^{m} (1+\alpha_{ij}) \in
		\bigoplus_{i=1}^{k} H^\ast(BT;\mathbb{Z}) \cong H^\ast_T(M^T;\mathbb{Z}).
	\]

\end{prop}

\begin{proof}
	$V$ restricted over a fixed point $p_i \in M^T$ is a $T$-equivariant complex
	vector bundle over a point. The proposition follows now from Example
	\ref{ex:equivarian char classes of trivial vb} and the naturality of characteristic classes.
\end{proof}

\begin{prop}\label{prop:char classes and everything else} 
	Under the assumptions of \ref{prop:equivariant chern classes for
		gkm} let $\alpha_{ij}$ denote the isotropy representation of $T$ in the
	corresponding tangent spaces of $M$. Then we have:
	\begin{enumerate}
		\item The image of the total equivariant Pontrjagin class of the action is
		      \[
			      \sum_{i=1}^k \prod_{j=1}^m (1+\alpha_{ij}^2)\in \bigoplus_{i=1}^k H^*(BT;\ZZ)\cong H^*_T(M^T;\ZZ).
		      \]
		\item The image of the total equivariant Stiefel-Whitney class of the action is
		      \[
			      \sum_{i=1}^k \prod_{j=1}^m (1\pm \alpha_{ij})\in \bigoplus_{i=1}^k H^*(BT;\ZZ_2)\cong H^*_T(M^T;\ZZ_2).
		      \]

	\end{enumerate}
\end{prop}
\begin{proof}
	The first assertion follows from Proposition \ref{prop:equivariant chern
		classes for gkm} and the fact, that Pontryagin classes are defined as the Chern
	classes of the complexified tangent bundle. A similar argument to that of
	Proposition \ref{prop:equivariant chern classes for gkm} applies to the
	Stiefel-Whitney classes.
\end{proof}

\begin{rem}
	With the term \emph{GKM theory} one usually denotes results related to the
	correspondence between GKM actions and GKM graphs, such as
	Proposition \ref{prop:char classes and everything else} and its consequences. Many different variants of the GKM theory described in this survey have been proposed, for example for a nondiscrete fixed point set, actions without fixed points, or actions of general compact connected Lie groups; see the end \cite[Section 11]{eqcohomsurvey} for a detailed list of references.
\end{rem}

\subsection{The abstract notion of a GKM graph and geometric structures}\label{sec:geometricstuff}

We have seen in Section \ref{sec:associatinggraph} how a torus action on a manifold can give rise to a labelled unoriented graph $(\Gamma,\alpha)$ with labels in $\mathbb{Z}_\mathfrak{t}^*/{\pm 1}\cong\mathbb{Z}^k/{\pm 1}$. When studying the interplay between combinatorics and geometry one is naturally confronted with the problem of finding the ``right'' combinatorial object. In this section we will study the rich interplay between varying degrees of structure -- from basic topology to equivariant Kähler structures -- and the combinatorics of the arising graphs. In particular we will fix terminology for the abstract combinatoric objects meant to describe the geometric situations.

For a graph $\Gamma$ we denote by $V(\Gamma)$ its set of vertices and by $E(\Gamma)$ its set of unoriented edges. Occasionally it will be convenient to consider the set $\widetilde{E}(\Gamma)$ of oriented edges containing two oriented edges for every unoriented edge in $E(\Gamma)$. For $e\in \widetilde{E}(\Gamma)$ we denote by $\overline{e}$ the same edge with the opposite orientation. There are two functions $i,t\colon \widetilde{E}(\Gamma)\rightarrow V(\Gamma)$ associating to each edge $e$ its initial vertex $i(e)$ and its terminal vertex $t(e)$. For some $v\in V(\Gamma)$ we set $E_v$ and $\widetilde{E}_v$ to be the unoriented and oriented edges emanating from $v$ (with orientation away from $v$).

\paragraph{Topology of the graph} As a first observations we note that the graph associated to a $2n$-dimensional GKM manifold is $n$-regular, i.e.\ every vertex is connected to exactly $n$ edges. To see this recall that for a fixed point $p$ the $2$-dimensional tangent representation $T_pM$ decomposes into $2$-dimensional subrepresentations and each of these corresponds to one edge leaving the vertex corresponding to $p$. Furthermore as any invariant $2$-sphere contains two distinct fixed points, the graph contains no single edge loops. Multiple edges between vertices do occur. As a less elementary observation we point out that the graph is always connected: indeed it follows directly from the Chang-Skjelbred lemma, Theorem \ref{thm:csl over Q},  that $\dim H^0(M_1,\mathbb{Q})= \dim H^0(M;\mathbb{Q})$. Hence $M_1$ is connected and consequently the same holds for the associated GKM graph $(\Gamma,\alpha)$.

\paragraph{Smoothness.} While the above conditions on $\Gamma$ follow from
purely topological arguments, the smoothness of the action results in a first
nontrivial condition on the labels. In fact the labels on the edges emanating
from two neighbouring vertices are closely related. We will very briefly sketch
the reasoning behind this: let $p,q\in M^T$ be two fixed points connected by a
$T$-invariant two-sphere $S$ and let $H\subset T$ denote the codimension $1$
subtorus contained in the kernel of the $T$-action on $S$.
Both $T_p M$ and $T_q M$ appear in the restriction of $TM$ to the trivial
$H$-space $S$. As the $H$-weights of the fibers of this
bundle vary continuously and are hence constant it follows  that $T_p M$ and
$T_q M$ are isomorphic as $H$-representations. This means that the weights at
$p$, which are encoded in the labels of the edges emanating from the
corresponding vertex, correspond bijectively -- albeit not
canonically -- to the corresponding weights at $q$ when restricting to the
subgroup $H$. Combinatorially such a bijection is captured
in the notion of a \emph{connection} in the definition below. It was first
considered in \cite{GuilleminZaraI} and \cite{GuilleminZaraII}. We picture it as
a method to slide edges along edges.

\begin{figure}[h]
	\centering
	\includegraphics{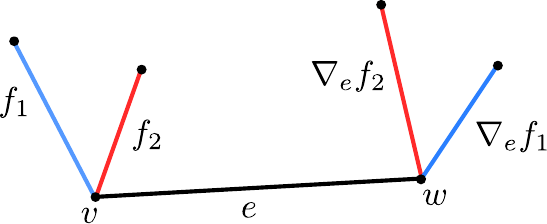}
	\caption{A connection}
	\label{fig: connection}
\end{figure}

\begin{defn}\label{defn:abstractgkmgraph}
	Let $\Gamma$ be a finite, connected, $n$-regular, loop-free graph together with a function $\alpha\colon E(\Gamma)\rightarrow \mathbb{Z}^k/{\pm 1}$, called an \emph{axial function}. A \emph{compatible connection} on $(\Gamma,\alpha)$ is given by a bijection $\nabla_e: \widetilde{E}_{i(e)}\rightarrow \widetilde E_{t(e)}$ for every oriented edge $e\in \widetilde{E}(\Gamma)$ such that
	\begin{enumerate}
		\item $\nabla_e(e)=\overline{e}$.
		\item $\nabla_{\overline{e}}=\nabla_e^{-1}$.
		\item $\alpha(\nabla_e(f))\equiv \alpha(f)\mod \alpha(e)$ in $\mathbb{Z}^k/{\pm 1}$ for all $f\in \widetilde{E}_{i(e)}$.
	\end{enumerate}
	The pair $(\Gamma,\alpha)$, consisting of the graph $\Gamma$ and the axial function $\alpha$, is called an \emph{abstract GKM graph} if it admits a compatible connection.
\end{defn}

We also occasionally like to use the term unsigned GKM graph in order to separate it from the original definition in \cite{GuilleminZaraI, GuilleminZaraII} which corresponds to what we will define as a signed GKM graph in the subsequent section on almost complex structures. Note that the congruence condition in the definition makes sense in $\ZZ^k/{\pm 1}$ although summation is not well defined. Equivalently, one may formulate it as follows: for any sign choices of $\alpha(\nabla_e(f))$, $\alpha(e)$, and $\alpha(f)$ there exists $c\in \ZZ$ and $\varepsilon\in\{\pm 1\}$ such that $\alpha(\nabla_e(f)) = \varepsilon \alpha(f) + c\alpha(e)$. The smoothness condition is related to the integrality of the linear combination; one may generalize GKM theory for actions on orbifolds (see \cite{GuilleminZaraII, GoertschesWiemelerNonNeg}), where rational constants occur.

The existence of a compatible connection on the GKM graph associated to a GKM manifold was shown in \cite[p.\ 6]{GuilleminZaraII} and \cite[Proposition 2.3]{GoertschesWiemelerNonNeg}. We thus obtain:
\begin{prop}\label{prop:abstract graph}
	The GKM graph of a GKM manifold is an abstract GKM graph.
\end{prop}

	As our definition of GKM manifold includes effectivity of the action, the GKM graph of a GKM manifold automatically satisfies the following condition:
	\begin{defn}
		An abstract GKM graph $(\Gamma,\alpha)$ is called \emph{effective} if for one and hence all $v\in V(\Gamma)$, the common kernel of the labels of all edges in $E(\Gamma)_v$ (considered as homomorphisms $T\to S^1$) is trivial.
	\end{defn}
	From now on we will assume that also all occurring abstract GKM graphs are effective.\\

Additional geometric structure on a GKM manifold $M$ leave its mark on the GKM graph. This is elaborated on in \cite[Section 2]{GKZ}.

\paragraph{Orientability.} In the literature GKM manifolds are usually assumed to be orientable. While this ensures that the one-skeleton is a graph of two-spheres and does not contain any copies of $\mathbb{R}P^2$, the orientability condition does place further restrictions on which labelled graphs can actually occur. Combinatorial notions of orientability are discussed in \cite[Section 7.9]{ToricTopology} in the special case of a $T^n$-action on a $2n$-dimensional manifold. A cohomologcal definition appeared in \cite{oddGKM}. Recently in \cite{GKZcorrespondence} we considered the following obstruction to the labels: given an abstract GKM graph $(\Gamma,\alpha)$, choose an arbitrary sign of the weight of every (unoriented) edge as well as a compatible connection. For an edge $e$ emanating from a vertex $v$ let $e_2,\ldots,e_n$ denote the other edges in $E_v$. Then through the fixed sign for $i=2,\ldots,n$ there are unique $\epsilon_i\in\{\pm 1\}$ satisfying the congruence
\[\alpha(\nabla_e(e_i))\equiv \epsilon_i\alpha(e_i)\mod \alpha(e)\]
in $\mathbb{Z}^k$. We set $\eta(e)=-\epsilon_2\cdot\ldots\cdot\epsilon_n$.
\begin{defn}
	We call $(\Gamma,\alpha)$ \emph{orientable} if there is a choice of signs for the labels and a compatible connection such that for every closed edge path $e_1,\ldots,e_m$ in $\Gamma$ one has \[\prod_{i=1}^m \eta(e_i)=1.\]
\end{defn}

In \cite{GKZcorrespondence} it is shown that the GKM graph of an (orientable) GKM manifold is indeed orientable in the above sense. Furthermore it is observed that for an orientable GKM graph any choice of signs and compatible connection will fulfil the condition in the definition.

\begin{ex} Let $\alpha,\beta,\gamma\in \mathbb{Z}^2$ be any linearly independent elements and consider the GKM graphs
	\begin{center}
		\begin{tikzpicture}
			\node (a) at (0,0)[circle,fill,inner sep=2pt] {};
			\node (b) at (3,0)[circle,fill,inner sep=2pt] {};
			\node (c) at (0,2)[circle,fill,inner sep=2pt] {};
			\node (d) at (3,2)[circle,fill,inner sep=2pt] {};

			\draw [very thick](a) to[in=160, out=20] (b);
			\draw [very thick](c) to[in=160, out=20] (d);
			\draw [very thick](a) to[in=-160, out=-20] (b);
			\draw [very thick](c) to[in=-160, out=-20] (d);
			\draw [very thick](a) to (c);
			\draw [very thick](b) to (d);

			\node at (1.5,-0.6) {$\alpha$};
			\node at (1.5,0.6) {$\beta$};
			\node at (1.5,1.4) {$\alpha$};
			\node at (1.5,2.6) {$\beta$};
			\node at (-0.4,1) {$\gamma$};
			\node at (3.4,1) {$\gamma$};

			\node (e) at (6,0)[circle,fill,inner sep=2pt] {};
			\node (f) at (9,0)[circle,fill,inner sep=2pt] {};
			\node (g) at (6,2)[circle,fill,inner sep=2pt] {};
			\node (h) at (9,2)[circle,fill,inner sep=2pt] {};

			\node at (7.5,-0.4) {$\alpha$};
			\node at (7.5,0.6) {$\beta$};
			\node at (7.5,1.4) {$\beta$};
			\node at (7.5,2.4) {$\alpha$};
			\node at (5.4,1) {$\gamma$};
			\node at (9.4,1) {$\gamma$};

			\draw [very thick](e) to (f);
			\draw [very thick](f) to (g);
			\draw [very thick](g) to (h);
			\draw [very thick](h) to (e);
			\draw [very thick](e) to (g);
			\draw [very thick](h) to (f);

		\end{tikzpicture}
	\end{center}
	which arise from each other by interchanging the endpoints of the two $\beta$-edges. There is a unique connection which preserves the labels and is thus in particular a compatible connection. We observe that with respect to this connection and given sign choices we have $\eta(e)=-1$ for any edge in the two graphs. Note that in the left hand graph the length of any closed edge path is even, while on the right hand side there are paths of odd length. Hence the left hand graph is orientable while the right hand one is not. Indeed the first graph is the GKM graph of a suitable action on $S^2\times S^4$, while the second graph encodes the one-skeleton of a suitable action on the non-orientable $(S^2\times S^2\times S^2)/\mathbb{Z}_2$ where $\mathbb{Z}_2$ acts by reflecting each $S^2$ factor at the equator.
\end{ex}

\paragraph{Almost complex structures.} Most prominently, one may consider actions on almost complex manifolds. If we are given a $T$-invariant almost complex structure on $M$, then the weights of the isotropy representations at the fixed points are well-defined elements in $H^2(BT;\ZZ)\cong \ZZ_\mft^*$. Thus we may associate to an oriented edge $e$, the weight (with unique sign) at the fixed point corresponding to the initial vertex $i(e)$. The arising combinatorial structure is given by the following
\begin{defn}
	A \emph{signed structure} on an abstract GKM graph $(\Gamma,\alpha)$ is a lift of the axial function to a map $\widetilde{E}(\Gamma)\rightarrow \mathbb{Z}^k$, by abuse of notation usually still denoted $\alpha$, such that
	\begin{enumerate}
		\item $\alpha(\overline{e})=-\alpha(e)$.
		\item there is a compatible connection fulfilling the signed congruence relations
		      \[ \alpha(\nabla_e(f))\equiv \alpha(f)\mod\alpha(e)\] in $\mathbb{Z}^k$ for all oriented edges $e\in\widetilde{E}(\Gamma)$ and $f\in \widetilde{E}_{i(e)}$.
	\end{enumerate}
	An abstract GKM graph together with a signed structure will also be referred to as a \emph{signed GKM graph}.
\end{defn}

The original definition of an abstract GKM graph given in \cite{GuilleminZaraI, GuilleminZaraII} was the signed version defined above. We note that an unsigned GKM graph may carry different signed structures, which on the geometric side corresponds to different invariant almost complex structures compatible with the same action.
We may regard the existence of a signed structure on the GKM graph of a GKM action as an obstruction to the existence of an invariant almost complex structure.

\begin{ex}\label{ex:spheres} It is a classical result by Borel and Serre \cite{BorelSerre} that the only spheres admitting any almost complex structure are $S^2$ and $S^6$; see also \cite{KonstantisParton}. This is nicely reflected on the level of GKM graphs. Any GKM action on $S^4$ has exactly two fixed points due to the Euler characteristic. Hence its GKM graph looks like the left hand graph below for some weights $\alpha,\beta$. The graph on the right depicts the GKM graph of the $T^2$-action on $S^6$ given by $(s,t)\cdot(v,w,z,h)=(s^{-1}v,st^{-1}w,tz,h)\in S^6\subset \mathbb{C}^3\oplus \mathbb{R}$. We have oriented the edges from left to right.

	\tikzset{middlearrow/.style={
				decoration={markings,
						mark= at position 0.5 with {\arrow{#1}} ,
					},
				postaction={decorate}
			}
	}

	\begin{center}
		\begin{tikzpicture}

			\node (a) at (0,0)[circle,fill,inner sep=2pt] {};
			\node (b) at (5,0)[circle,fill,inner sep=2pt]{};
			\node at (2.5,.85) {$\alpha$};
			\node at (2.5,-.95) {$\beta$};

			\draw [middlearrow={>}, very thick](a) to[in=160, out=20] (b);
			\draw [middlearrow={>}, very thick](a) to[in=200, out=-20] (b);

			\node (c) at (7,0)[circle,fill,inner sep=2pt] {};
			\node (d) at (12,0)[circle,fill,inner sep=2pt]{};
			\node at (9.5,1.1) {$-e_1$};
			\node at (9.5,0.3) {$e_1+e_2$};
			\node at (9.5,-1.1) {$-e_2$};

			\draw [middlearrow={>}, very thick](c) to[in=150, out=30] (d);
			\draw [middlearrow={>}, very thick](c) to[in=-150, out=-30] (d);
			\draw [middlearrow={>}, very thick](c) to (d);

		\end{tikzpicture}
	\end{center}
\end{ex}

One quickly checks that the connection where transport along any edge swaps the other two edges makes the right hand side a signed GKM graph. Indeed $S^6$ admits a $G_2$-invariant almost complex structure and the depicted graph is the GKM graph of the subaction of the maximal torus (cf.\ \cite[Example 1.9.1]{GuilleminZaraII}) . However on the graph associated to $S^4$ no signed structure can exist: let $e$ and $f$ denote the upper and lower edge oriented from left to right. Any connection has $\nabla_e(e)=\overline{e}$ and consequently also $\nabla_e(f)=\overline{f}$. However the signed congruence relations would imply
\[-\beta=\alpha(\nabla_e(f))\equiv \alpha(f)=\beta\mod \alpha\]
which cannot happen for linearly independent $\alpha$ and $\beta$. Although we know that the statement even holds non-equivariantly, this at least shows that there is no almost complex structure compatible with any GKM action on $S^4$.

\paragraph{Symplectic structures.} As a next step in the hierarchy we consider Hamiltonian $T$-actions.  Not every Hamiltonian action is GKM -- e.g.\ the fixed point set does not have to be discrete. However Hamiltonian actions are automatically equivariantly formal (this was shown by several authors \cite{Duflot, AtiyahBott, Ginzburg, Kirwan}, see \cite[Example 7.9]{eqcohomsurvey} for a sketch of the argument) and thus have a natural home in GKM theory. As sketched in Remark \ref{rem:Hamiltonian},  the image $\mu(M_1)$ of the one-skeleton is a linear presentation of the GKM graph, i.e.\ an edgewise affine linear embedding into $\mathfrak{t}^*$ such that the label of each edge gives its slope as in Examples \ref{ex:generaltoric}, \ref{ex:flag}, and \ref{ex:eschenburg}. An example of a signed GKM graph which does not admit such an embedding is given by the $T^2$-action on $S^6$ from Example \ref{ex:spheres}. However there are further obstructions to being the GKM graph of a Hamiltonian action: by the famous convexity theorem of Atiyah \cite{Atiyah} and Guillemin--Sternberg \cite{GSconvexity} the image of $\mu$ is a convex polytope given by the convex hull of the images of the fixed points. Furthermore it can be seen from the local normal form of the moment map that the orbits over the edges of this polytope are of dimension $\leq 1$ and hence belong to the linearly embedded GKM graph. We arrive at the following \emph{convexity criterion}: the GKM graph of a Hamiltonian GKM action has a linear presentation such that the edges of its convex hull must be given by edges of the graph.

\begin{rem}
	A definition similar to the GKM graph was given in \cite{Tolman}. Tolman defined
	the \emph{x-ray} of a Hamiltonian action of a torus $T$ on a compact symplectic
	manifold $M$. Consider an isotropy group $H = T_p$ and let $N$ be a connected
	component of the fixed point set $M^H$. As $N$ is itself a symplectic manifold
	equipped with a Hamiltonian action of a torus, its image under the moment map is
	a convex polytope contained in the convex polytope of the moment image of $M$.
	Moreover it is the convex hull of all $T$-fixed points that are contained in
	$N$. The x-ray is the collection of all these convex polytopes for every
	emerging isotropy group $H$ and every connected component of its fixed point
	set.

	If the Hamiltonian torus actions is also GKM, then the image of fixed points
	in the x-ray correspond to the vertices of the GKM graph and the
	one-dimensional polytopes correspond to a two-sphere in the one-skeleton, hence
	to the edges of the GKM graph. Moreover the slopes of the polytope represent a
	positive multiple of the corresponding weight in the isotropy representation,
	hence the x-ray determines the labels of the GMK graph up to a positive
	multiple.
\end{rem}

\begin{ex}
	Consider the GKM graphs with the linear presentations
	\begin{center}
		\begin{tikzpicture}
			\draw[step=1, dotted, gray] (-2.5,-2.5) grid (8.5,1.5);

			\draw[very thick] (0,0) -- ++(1,0) -- ++(0,-1) -- ++(-2,0) -- ++(0,2) -- ++(3,0) -- ++(0,-3) -- ++(-2,0) -- ++(0,2);

			\node at (0,0)[circle,fill,inner sep=2pt] {};
			\node at (1,0)[circle,fill,inner sep=2pt] {};
			\node at (1,-1)[circle,fill,inner sep=2pt] {};
			\node at (-1,-1)[circle,fill,inner sep=2pt] {};
			\node at (-1,1)[circle,fill,inner sep=2pt] {};
			\node at (2,1)[circle,fill,inner sep=2pt] {};
			\node at (2,-2)[circle,fill,inner sep=2pt] {};
			\node at (0,-2)[circle,fill,inner sep=2pt] {};

			\draw[very thick] (4,1) -- ++(1,0) -- ++(0,-2) -- ++(1,0) -- ++(0,2) -- ++(1,0) -- ++(0,-1) -- ++(-3,0) -- ++(0,1);
			\node at (4,1)[circle,fill, inner sep=2pt] {};
			\node at (5,1)[circle,fill, inner sep=2pt] {};
			\node at (5,-1)[circle,fill, inner sep=2pt] {};
			\node at (6,-1)[circle,fill, inner sep=2pt] {};
			\node at (6,1)[circle,fill, inner sep=2pt] {};
			\node at (7,1)[circle,fill, inner sep=2pt] {};
			\node at (7,0)[circle,fill, inner sep=2pt] {};
			\node at (4,0)[circle,fill, inner sep=2pt] {};
		\end{tikzpicture}
	\end{center}
	such that the labels are the primitive integral representatives of the slopes. Considering oriented edges with oriented slopes as labels gives rise to signed structures. Note that as unsigned graphs, both graphs are equal to an $8$-gon with labels alternating between $\pm e_1$ and $\pm e_2$, and are hence realized by the quasitoric manifold given by the equivariant connected sum of three copies of $S^2\times S^2$, where every circle factor of the acting $T^2$ rotates one of the $S^2$ factors at unit speed. See also Section \ref{sec:dim4} below. The above signed GKM graphs are realized by two invariant almost complex structures on this manifold which are not equivariantly isomorphic.

	The boundary of the convex hull of the above linear presentations contains an edge which is not in the graph and hence the embedding cannot come from a moment map of a Hamiltonian $T^2$-action. In fact one can show that no linear presentation of the underlying unsigned GKM graph can satisfy the convexity criterion and hence the action does not admit a compatible symplectic structure making it Hamiltonian.
\end{ex}

\paragraph{K\"ahler structures.}

The final geometric structure we want to discuss is that of an invariant Kähler structure, i.e.\ a Hamiltonian action with an invariant (integrable) complex structure compatible with the symplectic form. There are further combinatorial obstructions in order for a $T$-manifold $M$ to admit such a structure as shown by Tolman in \cite[Lemma 3.5]{Tolman}: for $p\in M^T$ and $V\subset T_p M$ a sum of irreducible subrepresentations whose weights form the minimal generating set of a cone in $\mathfrak{t}^*$ there is a Hamiltonian $T$-submanifold $N$ containing $p$ such that $T_pN=V$. In particular the moment map embeds the GKM graph of $N$ as a subgraph of the surrounding GKM graph of $M$ such that at the vertex representing $p$, the edges of the subgraph are precisely those belonging to $V$. Now the additional combinatorial obstruction comes from the fact that the embedded subgraph fulfils the convexity criterion as well.

\begin{ex}\label{ex:nonkahlerness}
	We illustrate this with Eschenburgs twisted flag and the standard flag from Examples \ref{ex:flag} and \ref{ex:eschenburg}.

	\begin{center}
		\begin{tikzpicture}
			\draw[step=1, dotted, gray] (-3.5,-4.5) grid (9.5,2.5);

			\draw[fill=cyan, opacity=0.4] (1,1) -- ++ (-2,0) -- ++ (2,-2) -- ++ (0,2);
			\draw[fill=cyan, opacity=0.4] (4,1) -- ++(0,-1) -- ++ (3,-3) -- ++ (0,4) -- ++ (-3,0);

			\draw[very thick] (1,1) -- ++(-3,0) -- ++(4,-4) -- ++(0,3) -- ++(-1,1);
			\draw[very thick] (2,-3)--++(-1,+2)--++(0,2)++(0,-2)--++(-1,1)--++(2,0)++(-2,0)--++(-2,1);
			\node at (1,1)[circle,fill,inner sep=2pt]{};

			\node at (-2,1)[circle,fill,inner sep=2pt]{};

			\node at (1,-1)[circle,fill,inner sep=2pt]{};

			\node at (0,0)[circle,fill,inner sep=2pt]{};

			\node at (2,0)[circle,fill,inner sep=2pt]{};

			\node at (2,-3)[circle,fill,inner sep=2pt]{};

			\draw[very thick] (7,1) -- ++(-3,0) -- ++(4,-4) -- ++(0,3);
			\draw[very thick] (8,0) -- ++(-1,1);
			\draw[very thick] (8,-3)--++(-1,0);
			\draw[very thick] (7,-3)--++(0,4);
			\draw[very thick] (7,-3) -- ++(-3,3) -- ++ (4,0);
			\draw[very thick] (4,0) -- ++(0,1);

			\node at (7,1)[circle,fill,inner sep=2pt]{};

			\node at (4,1)[circle,fill,inner sep=2pt]{};

			\node at (7,-3)[circle,fill,inner sep=2pt]{};

			\node at (4,0)[circle,fill,inner sep=2pt]{};

			\node at (8,0)[circle,fill,inner sep=2pt]{};

			\node at (8,-3)[circle,fill,inner sep=2pt]{};

		\end{tikzpicture}
	\end{center}
	In the moment image of the Eschenburg flag, there is no convex subgraph containing the two edges spanning the shaded cone. Hence this moment image cannot come from a Kähler action. On the contrary every such cone generated by adjacent edges in the moment image of the standard flag is part of a convex subgraph. Proving that a $T$-manifold belonging to the left hand GKM graph does not admit a compatible Kähler structure requires more work: one needs to systematically apply the above criterion to all possible linear presentation compatible with the underlying unsigned GKM graph. This is Tolman's strategy  for constructing Hamiltonian non-Kähler actions in \cite{Tolman} and the left hand example is, at least combinatorially, precisely Tolman's example. We will discuss the relations in more detail in Example \ref{ex:twe}. We shall see more examples of this type in the discussion of GKM fibrations in Section \ref{sec:dim6}.
\end{ex}

\subsection{The GKM correspondence} \label{sec:GKMcorrespondence}

There are several natural notions of isomorphism of abstract GKM graphs. The easiest one is the following:

\begin{defn}\label{defn:GKMgraphiso}
	An \emph{isomorphism} $\varphi:\Gamma\to \Gamma'$ of abstract GKM graphs is an isomorphism  of abstract graphs such that
	\[
		\alpha(\phi(e)) = \alpha(e)
	\]
	for all edges $e$ of $\Gamma$.
\end{defn}
Any isomorphism of the GKM graphs of GKM actions on $M$ and $N$ induces a map between the corresponding equivariant cohomology algebras which is an $H^*(BT)$-algebra isomorphism. Sometimes one allows, in the definition of an isomorphism of abstract GKM graphs, an additional automorphism $\psi:T\to T$ of the torus such that $\alpha(\phi(e)) = \alpha(e)\circ (d\psi)^*$. In this case, the induced map on equivariant cohomology is only a twisted isomorphism.

The \emph{GKM correspondence} is the map
\[
	\left\{ \textrm{GKM manifolds} \right\}\longrightarrow \left\{  \textrm{abstract (orientable) GKM graphs} \right\}
\]
which we consider as a generalization of the well-known correspondences in toric geometry sending a toric symplectic manifold to its Delzant polytope \cite{Delzant}, a toric variety to its fan \cite[Chapter 5]{ToricTopology}, or a quasitoric manifold to its orbit space and its characteristic function \cite{DavisJ}. In view of the additional structure imposed by invariant geometric structures on a GKM manifold one may also consider variants of this correspondence, such as e.g.\
\[
	\left\{ \textrm{almost complex GKM manifolds} \right\}\longrightarrow \left\{  \textrm{abstract signed GKM graphs} \right\}.
\]
Note also that these correspondences may be investigated for integer as well as rational GKM manifolds; it becomes apparent that the term \emph{GKM correspondence} in fact comprises a whole collection of related correspondences between geometric and combinatorial objects.

Two natural questions arise:
\begin{itemize}
	\item The \emph{GKM realization} question asks which abstract GKM graphs are realizable by GKM manifolds.
	\item The \emph{GKM rigidity} question asks whether two simply-connected GKM manifolds with isomorphic GKM graphs are isomorphic, where the isomorphism can be understood in various different ways such as (twisted) (equivariant) homotopy equivalence, homeomorphism or diffeomorphism.
\end{itemize}
Again, both questions may be investigated in the smooth category or in presence of additional geometric structures. For rigidity questions it seems reasonable to restrict to the integer GKM case unless one is willing to relax the notion of equivalence accordingly (e.g.\ to rational equivalence in the sense of rational homotopy theory). Even though GKM manifolds have vanishing odd cohomology, the additional assumption of simply-connectedness in the rigidity question is necessary: for instance one can change the fundamental group without violating the GKM conditions by modifications in the regular stratum related to taking connected sum of the orbit space with homology spheres, see \cite[Theorem 2.4]{WiemelerExotic}.

\begin{rem}\label{rem:rigidity}
	One motivation for the rigidity question is given by the fact that the GKM graph of a GKM manifold determines many of its topological properties, see Proposition \ref{prop:char classes and everything else}. Note that in order to obtain this description of topological invariants one only needs a torus of dimension at least $2$ independent of the dimension of the manifold. Thus this powerful result holds independently of the so-called complexity of the action, which for a $T^k$-action on a $2n$-dimensional manifold is defined as the number $n-k$. In the presence of a fixed point it is always non-negative. Although independence of the complexity is part of the appeal of the GKM setup, the complexity does play an important role in the context of rigidity questions. Actions below are not assumed to be GKM as the mentioned results are not specifically for the GKM setup (although related). The complexity $0$ case, i.e.\ $T^n$ acting on $M^{2n}$, is very well understood and often completely encoded in the combinatorics of the action. A nice introduction and overview is given in \cite[Chapter 7]{ToricTopology}. More recently there has been progress in complexity $1$: although the situation is more complicated than the complexity $0$ case, a rigidity result for certain complexity $1$ actions with connected stabilizers has been obtained in \cite[Theorem 5.5]{Ayzenberg} characterizing them up to equivariant homeomorphism through certain characteristic data. For complexity $\geq 2$, rigidity questions seem to be wide open.
\end{rem}

Concerning the realization question one observes directly that because of \eqref{eq:cohomiso} the  structure of the equivariant cohomology of a GKM action immediately imposes additional restrictions on the abstract GKM graphs that are realizable by a GKM manifold.
\begin{prop}\label{prop:necessaryforrealization}
	If an abstract GKM graph $(\Gamma,\alpha)$ is realizable by an rational GKM manifold, then
	\begin{enumerate}
		\item its equivariant graph cohomology $H^*_T(\Gamma,\alpha;\QQ)$ is a free $H^*(BT;\QQ)$-module.
		\item its graph cohomology $H^*(\Gamma,\alpha;\QQ)$ satisfies Poincar\'e duality.
	\end{enumerate}
\end{prop}

	An integral version of the same statement as in Proposition \ref{prop:necessaryforrealization} hold true if one imposes assumptions on the isotropy groups of the realizing action as in Proposition \ref{prop:char classes and everything else}. 
	
\begin{rem}The conditions on an abstract GKM graph that turned out to be necessary for realizability in the above proposition may be added on the combinatorial side of the GKM correspondence. However we point out that the relations between the abstract combinatorial conditions are in need of further study, in particular when it comes to conditions that concern the graph cohomology.
	We showed in \cite{GKZcorrespondence} that for an abstract $T^2$-GKM graph freeness of rational equivariant graph cohomology is automatic. We do not know about graphs for higher rank tori or the situation with integer coefficients.

	Poincar\'e duality of the graph cohomology is a non-void condition. For example, one may consider $T$-actions on manifolds that satisfy all conditions of a GKM actions except orientability, such that the one-skeleton is still a union of two-spheres, to obtain abstract GKM graphs that do not satisfy Poincar\'e duality, see \cite[Example 2.22]{GKZcorrespondence}.

	In \cite{Solomadin} Solomadin constructed certain GKM graphs that are not realizable by a GKM manifold; it would be interesting to investigate their (equivariant) graph cohomology in light of these properties.
\end{rem}

In Sections \ref{sec:dim4}, \ref{sec:dim8}, and \ref{sec:dim6} we will elaborate on the known results concerning both the realization and the rigidity question in relation to the dimension of the manifold. Briefly, the situation is well understood in dimension $\leq 4$ and several versions of GKM rigidity will turn out to be true in dimension $6$, but not in dimensions $8$ and higher. Realization results exist in dimension $\leq 6$; the situation in dimensions $8$ and higher is wide open.

\section{Dimension $\leq 4$}\label{sec:dim4}

In dimension $2$, a compact orientable surface $M$ admitting a circle action is necessarily $S^2$, by the same argument as for the proof of Proposition \ref{prop:Ntwospheres}. In fact, if the action is effective, it is necessarily equivariantly diffeomorphic to the standard action on $S^2$, which can be seen as follows: the orbit space of the action is homeomorphic to a closed interval, with the only two singular orbits corresponding to its end points. It follows that $M$ is an equivariant gluing of two discs along there boundary, which results in the standard action on $S^2$. This sketch of an argument is also the easiest special case of the correspondence between cohomogeneity-one actions and their group diagrams, see for example \cite[Theorem 7.1]{Alekseevsky2}.

Given a GKM action of a torus $T$ on a $4$-dimensional simply-connected manifold $M$, linear independence of the weights in the fixed points implies that the dimension of $T$ is necessarily $2$. Hence, in this situation we are still in the case where the acting torus has half the dimension of the manifold, so that more classical theory applies. In fact, $M$ is what is usually called a \emph{torus manifold} \cite{MasudaUnitary, MasudaPanov}, i.e., a closed connected orientable even-dimensional smooth manifold with an action of a torus of half dimension with nonempty fixed point set.

As $M$ is simply-connected, by the universal coeffient theorem and Poincar\'e duality its odd integer cohomology vanishes, so that \cite[Theorem 2]{MasudaPanov} is applicable: the torus manifold $M$ is locally standard and its orbit space is face-acyclic. In other words, the orbit space $M/T$ is homeomorphic to a closed disc, in such a way that the orbit space $M_1/T$ of the one-skeleton of the action is the boundary of this disc, and every point in the interior has trivial isotropy. Note that in case the action has at least three fixed points, we may consider the orbit space as a two-dimensional convex polytope, hence we are in the realm of quasitoric manifolds \cite{DavisJ}. In this low dimension it was observed in \cite[Theorem on p.\ 537]{OrlikRaymondI} or \cite[Section 6]{WiemelerExotic} that the data of the face poset of the orbit space, together with the isotropy groups, determine $M$ not only up to equivariant homeomorphism but up to equivariant diffeomorphism. The latter reference argues that this is the case in dimension $6$, but the same proof works in dimension $4$. In higher dimension the problem is related to that of exotic smooth structures on discs. Also GKM realization works in this dimension: in case the graph has only $2$ vertices, it is realized by a linear action on $S^4$; in case of more vertices, one may regard the GKM graph as the boundary of a convey polytope, so that the classical construction of a quasitoric manifold with given orbit space and characteristic function \cite{DavisJ} applies.  We obtain a one-to-one correspondence

\begin{center}
	\begin{tikzpicture}
		\node (a) at (0,0) {$\left\{\begin{array}{c}
					\text{4-dim.\ simply-connected} \\
					\text{$T^2$-GKM manifolds}
				\end{array}\right\}$};

		\node (b) at (8.7,0) {$\left\{\begin{array}{c}
					\text{2-valent GKM graphs}\end{array}\right\}$};

		\draw[<->] (a) -- (b);
	\end{tikzpicture}
\end{center}
where the left hand side is considered up to equivariant diffeomorphism, and the right hand side up to isomorphisms of GKM graphs. More details on this as well as the claims below are provided in \cite[Section 5.1]{GKZ}.

Regarding the almost complex correspondence, any $2$-valent signed GKM graph is realized by a $4$-dimensional quasitoric manifold with invariant almost complex structure. This follows from \cite{Kustarev} because the signed structure yields what is called a positive omniorientation. We are not aware of a solution of the almost complex rigidity question.

In dimension $4$ the Hamiltonian and invariant Kähler case coincide and reduce to the classical Delzant correspondence \cite{Delzant}: the momentum image of such an action is a Delzant polytope, which immediately gives the GKM graph as all edges in the graph are visible on the boundary of the polytope. Conversely manifolds with the same momentum polytope (up to translation) are equivariantly symplectomorphic. We note that the datum of the GKM graph is not quite enough to obtain symplectic rigidity as the length of the edges in the momentum image is not encoded in the abstract graph but does play a role for the symplectic structures.
\section{Dimension $\geq 8$}\label{sec:dim8}

In \cite{GKZrigid}, we showed that in dimension $8$ the GKM rigidity question can be answered in the negative, by exhibiting two $8$-dimensional $T^3$-GKM manifolds with identical GKM graphs that do not have the same homotopy groups. By taking products, e.g.\ with factors of $S^2$, one then also obtains examples in higher dimensions.
\begin{thm}[{\cite[Theorem 1.1]{GKZrigid}}] \label{thm:rigidity kaputt}
	The total spaces of the two $\CC P^1$-bundles over $S^6$ admit $T^3$-GKM actions with identical GKM graphs.
\end{thm}
In fact, the actions we constructed have connected stabilizers and are even GKM$_3$, i.e., at any fixed point any three weights are linearly independent.

\begin{rem}
	As is discussed in more detail in \cite{GKZrigid}, rigidity up to non-equivariant homeomorphism is known whenever the complexity of the action (see Remark \ref{rem:rigidity}) is $0$ \cite[Theorem 3.4]{WiemelerTorusManifolds}, the number of fixed points is $\leq 3$ \cite{GKZrigid}, or the dimension is $\leq 6$ (see Theorem \ref{thm:dim6diffeotype} below). Thus the above example is minimal with respect to these three parameters.
\end{rem}

Let us briefly explain the idea of the construction: principal $\SU(2)$-bundles over $S^6$ are classified by homotopy classes of maps $S^6\to S^4 = {\mathbb{H}} P^1 \subset {\mathbb{H}} P^\infty$. As $\pi_6(S^4)=\ZZ_2$, there are precisely two such bundles, the trivial one induced by the constant map, and a nontrivial one induced by a homotopically nontrivial map $f:S^6\to S^4$: it is the pullback of the Hopf bundle $S^7\to S^4$ via $f$. Taking the fiberwise quotient by the maximal torus of $\SU(2)$ results in the trivial bundle $S^6\times \CC P^1$ and the pullback of the $\CC P^1$-bundle $\CC P^3\to S^4$ via $f$, whose total space we call $X$.
\[\xymatrix{
		X\ar[r]\ar[d] & \CC P^3\ar[d]\\
		S^6 \ar[r]^f & S^4
	}\]
One computes directly that $\pi_5(S^6\times S^2) = \pi_5(S^2)= \ZZ_2\neq 0 = \pi_5(X)$, so that these two spaces are not homotopy equivalent.

The nontrivial map $f$ a priori is unrelated to any group action; the main idea of the construction is to make $f$ equivariant with respect to the standard $T^3$-action on $S^6$ and a certain $T^3$-action on $S^4$. Then we lift the action on $S^4$ to the  bundle $\CC P^3\to S^4$ and consider the pullback action on $X$. The GKM graph of the actions constructed is given as follows.
\begin{center}
	\begin{tikzpicture}

		\node (a) at (0,0)[circle,fill,inner sep=2pt] {};
		\node (b) at (6,0)[circle,fill,inner sep=2pt]{};
		\node at (3,1) {$(1,0,0)$};
		\node at (3,0.3) {$(0,1,0)$};
		\node at (3,-1) {$(0,0,1)$};

		\draw [very thick](a) to[in=160, out=20] (b);
		\draw [very thick](a) to (b);
		\draw [very thick](a) to[in=200, out=-20] (b);

		\node (c) at (0,4)[circle,fill,inner sep=2pt] {};
		\node (d) at (6,4)[circle,fill,inner sep=2pt]{};
		\node at (3,5) {$(1,0,0)$};
		\node at (3,4.3) {$(0,1,0)$};
		\node at (3,3) {$(0,0,1)$};

		\draw [very thick](c) to[in=160, out=20] (d);
		\draw [very thick](c) to (d);
		\draw [very thick](c) to[in=200, out=-20] (d);

		\draw [very thick] (c) to (a);
		\draw [very thick] (d) to (b);

		\node at (7.2,2) {$(1,-1,-1)$};
		\node at (-1.2,2) {$(1,-1,-1)$};

	\end{tikzpicture}
\end{center}

We also obtained a Hamiltonian variant of this pair of actions. For this variant, not only the GKM graphs coincide but also their momentum images. Consider the map $k:\CC P^3\to \CC P^3/\CC P^2\cong S^6$ which is obtained by collapsing $\CC P^2=\{[z_0:\ldots : z_3]\mid z_3=0\}$ inside $\CC P^3$. This map is equivariant with respect to the standard $T^3$-action on $\CC P^3$ in the first three homogeneous coordinates, so that we can pull back our two bundles via $k$ to obtain two $\CC P^1$-bundles over $\CC P^3$, the trivial one $\CC P^3\times \CC P^1$ and a space $Y$, both equipped with GKM $T^3$-actions.
\[\xymatrix{
		Y \ar[r]\ar[d] & X\ar[r]\ar[d] & \CC P^3\ar[d]\\
		\CC P^3 \ar[r]^k & S^6 \ar[r]^f & S^4
	}\]
\begin{thm}[\cite{GKZrigid}] The $T^3$-spaces $\CC P^3\times \CC P^1$ and $Y$ satisfy:
	\begin{enumerate}
		\item They are not homotopy equivalent, as $\pi_6(Y) = \ZZ_6$ and $\pi_6(\CC P^3\times \CC P^1) = \ZZ_{12}$.
		\item They are GKM (in fact GKM$_3$) with the same GKM graph.
		\item Both admit a non-invariant K\"ahler structure.
		\item Both admit a $T^3$-invariant symplectic structure with compatible momentum maps whose x-rays coincide.
	\end{enumerate}
\end{thm}

The non-invariant K\"ahler structure on $Y$ exists, as $Y$ is the projectivization of a complex vector bundle over $\CC P^3$, see \cite[Remark 5.5]{GKZrigid}. We do not know if $Y$ admits an invariant K\"ahler structure. The invariant symplectic structure also makes use of this bundle structure, by applying a construction that goes back to Thurston \cite{Thurston}, see also \cite[Theorem 6.1.4]{McDuffSalamon}: one adds a form that restricts to a symplectic form on the fiber to a sufficiently large multiple of the pullback of the symplectic form on the base.

\section{Dimension $6$}\label{sec:dim6}

One of the most interesting playgrounds for low dimensional GKM theory is given by $T^2$-actions in dimension $6$. The reason is that this is the easiest setup in which GKM theory diverges from the classical (quasi)toric theory in complexity $0$ (cf.\ Remark \ref{rem:rigidity}). Hence it allows for both exotic examples as well as powerful results. A nice illustration of the first aspect is given by the spaces discussed in the next section.

\subsection{Exotic Hamiltonian actions and non-equivariant rigidity}

\begin{ex}\label{ex:twe} In \cite{Tolman}, Tolman constructed the first example of a simply-connected symplectic manifold admitting a Hamiltonian torus action with finitely many fixed points such that the action does not admit any invariant K\"ahler structure (which does not need to be compatible with the given symplectic structure). This is in fact the minimal dimension and complexity where this can occur since in dimension $4$ any Hamiltonian circle action with finitely many fixed points extends to a toric $T^2$-action \cite{Karshon} and a Hamiltonian $T^3$-action in dimension $6$ is toric as well.

	She constructed her example by symplectically gluing two Hamiltonian $6$-manifolds in order to arrive at a Hamiltonian $T^2$-action whose momentum image and GKM graph is the same as that of the $T^2$-action on the Eschenburg flag described in \ref{ex:eschenburg}. A third example of this kind was given in \cite{Woodward}, which is again a Hamiltonian $T^2$-action on a symplectic $6$-manifold that even extends to a multiplicity-free $\U(2)$-action. It has the same underlying combinatorics with regards to momentum image and GKM graph. The strategy how to see the non-Kählerness (in all three examples) from the GKM graph was outlined in Example \ref{ex:nonkahlerness}.

	Most interestingly, the Eschenburg flag -- as it arises as the projectivization of a rank two complex vector bundle over $\CC P^2$ -- admits a (non-invariant) K\"ahler structure (see \cite[Theorem 2]{Eschenburg}, \cite{EscherZiller}, or \cite[Section 4.2]{GKZsympbiquot}). In fact, more specifically, in the latter reference we argue that the $T^2$-equivariant symplectic form admits a compatible (integrable) complex structure. However by Tolman's arguments such a structure cannot be $T^2$-invariant. The question whether any (non-equivariant) Kähler structure exists on Tolman's and Woodward's construction is settled by Theorem \ref{thm:dim6diffeotype} below. It shows that all three examples are non-equivariantly diffeomorphic. Furthermore we note that it follows from Theorem \ref{thm:onetotonecorrespondence} that all three examples are $T^2$-equivariantly homeomorphic.
\end{ex}

Concerning the rigidity question, we observed in \cite{GKZdim6} that in this dimension, the GKM graph determines the non-equivariant diffeomorphism type.

\begin{thm}[{\cite[Theorem 3.1(b)]{GKZdim6}}]\label{thm:dim6diffeotype}
	Consider two $6$-dimensional simply-connected compact, connected smooth manifolds $M$ and $N$, equipped with $T^2$-actions of GKM type such that for all $p\notin M_1$ and $p\notin N_1$, the isotropy group $T_p$ is contained in a proper subtorus of $T$. Assume furthermore that there is an isomorphism $\varphi:\Gamma_M\to \Gamma_N$ between the associated GKM graphs (cf.\ Definition \ref{defn:GKMgraphiso}). Then there is a (non-equivariant) diffeomorphism $M\to N$ which induces the same homomorphism $H^*(N;\ZZ)\to H^*(M;\ZZ)$ as $\varphi$.
\end{thm}
Essentially, the proof consists of combining Proposition \ref{prop:char classes and everything else}, more precisely the fact that the GKM graph determines the integral cohomology ring as well as the Pontrjagin and Stiefel-Whitney classes, with the diffeomorphism classification of simply-connected $6$-manifolds due to Wall \cite{Wall}, Jupp \cite{Jupp} and \v{Z}ubr \cite{Zubr}.

\subsection{Realization of GKM fibrations with geometric
structure}\label{sec:gkm fibrations}

In \cite{GKZ} we took the phenomena in Example \ref{ex:twe} as a motivation to systematically construct many more exotic $6$-dimensional examples of this type, with arbitrary high number of fixed points. We realized certain GKM fibrations geometrically as projectivizations of complex rank $2$ vector bundles, and investigated the existence of various geometric structures on these realizations. To formulate the theorem we make use of the notion of a GKM fibration $\Gamma\to B$, which was introduced in \cite{GKMFiberBundles}. The definition simplifies for fibrations of $3$-valent graphs over $2$-valent ones which will be our focus here, so we refrain from giving the complete technical definition but explain it by means of a picture; see \cite[Section 3]{GKZ} for more details.
\begin{center}
	\begin{tikzpicture}
		\draw[step=1, dotted, gray] (-3.5,-4.5) grid (3.5,2.5);
		\draw[step=1, dotted, gray] (5.5,-4.5) grid (11.5,2.5);

		\draw[very thick] (-2,0) -- ++(1,1) -- ++(3,0) -- ++(0,-2) -- ++(-2,-2) -- ++(-2,0) -- ++(0,3);
		\draw[very thick] (-2,0) -- ++(3,0) -- ++(0,-1) -- ++(-1,-1) -- ++(-1,0) -- +(0,3);
		\draw[very thick] (-2,-3)--++(1,1);
		\draw[very thick] (1,0)--++(1,1);
		\draw[very thick] (1,-1)--++(1,0);

		\draw[very thick] (0,-2)--++(0,-1);

		\node at (-2,0)[circle,fill,inner sep=2pt]{};

		\node at (-1,1)[circle,fill,inner sep=2pt]{};

		\node at (2,1)[circle,fill,inner sep=2pt]{};

		\node at (2,-1)[circle,fill,inner sep=2pt]{};

		\node at (0,-3)[circle,fill,inner sep=2pt]{};

		\node at (-2,-3)[circle,fill,inner sep=2pt]{};
		\node at (-1,-2)[circle,fill,inner sep=2pt]{};

		\node at (0,-2)[circle,fill,inner sep=2pt]{};
		\node at (1,-1)[circle,fill,inner sep=2pt]{};
		\node at (1,0)[circle,fill,inner sep=2pt]{};

		\draw[->, very thick] (3.5,-1) -- ++(2,0);

		\draw[very thick] (7,1)  -- ++(3,0) -- ++(0,-2) -- ++(-1,-1) -- ++(-2,0) -- ++(0,3);

		\node at (7,1)[circle,fill,inner sep=2pt]{};

		\node at (10,1)[circle,fill,inner sep=2pt]{};

		\node at (10,-1)[circle,fill,inner sep=2pt]{};

		\node at (9,-2)[circle,fill,inner sep=2pt]{};

		\node at (7,-2)[circle,fill,inner sep=2pt]{};

	\end{tikzpicture}
\end{center}
As it is easier to describe the labels we give a picture that may also be interpreted symplectically: in these GKM graphs $\Gamma$ and $B$, the labels are given by the primitive vectors in direction of the slopes of the edges. The GKM fibration, indicated by the arrow, consists of a map on vertices and on edges. Vertices are sent to vertices as indicated by the picture, by collapsing appropriate edges. On edges the map is only defined partially: for every vertex $v$, a subset of the edges at $v$ is mapped bijectively to the edges at the image of $v$, in such a way that labels (slopes) are preserved. These edges are called horizontal, the remaining (collapsed) edges are called vertical. We also assume that there are connections on both graphs that are compatible in the sense that the connection on $\Gamma$ preserves vertical and horizontal edges and covers the connection on the base graph.

We wish to realize such a situation by a projectivization of an equivariant complex rank $2$ vector bundle over a $4$-dimensional $T^2$-manifold $X$ realizing $B$ (in the picture, this would be a blow-up of $\CC P^1\times \CC P^1$), such that the collapsed edges correspond to the $\CC P^1$ fibers. As $\CC P^1$ is naturally almost complex, it is natural to consider only \emph{fiberwise signed} fibrations, i.e., we assume that there is a signed structure, only partially defined on all vertical edges, which is respected by the connection on $\Gamma$ in the sense of an abstract signed GKM graph, when transporting vertical edges. If $\Gamma$ and $B$ are equipped with globally defined signed structures which are preserved by the GKM fibration, then we speak of a \emph{signed GKM fibration}. In our situation of a fibration of a $3$-valent graph fibering over a $2$-valent one, we furthermore observe that fibrations fall into one of two categories: they are either of \emph{product type} or of \emph{twisted type}, depending on whether graph-theoretically $\Gamma$ is a product of $B$ with a line or a ``M\"obius band'' over $B$. The example above is of twisted type.

\begin{thm}[\cite{GKZ}] \label{thm:gkm fibration} Consider a GKM fibration $\pi:\Gamma\to B$, where $\Gamma$ and $B$ are integer GKM graphs, $\Gamma$ of valency $3$ and $B$ of valency $2$.
	\begin{enumerate}
		\item If $\pi$ is fiberwise signed, then $\pi$ is geometrically realized as the projectivization $\PP(E)$ of a $T$-equivariant rank $2$ complex vector bundle $E\to X$ over a four-dimensional $T^2$-manifold $X$ which can be taken to be $S^4$ if $B$ has two vertices and quasitoric if $B$ has at least $3$ vertices.

		\item If $\pi$ is signed, then its geometrical realization as in (i) can be chosen to be a fibration of almost complex manifolds such that the induced fibration of signed GKM graphs is $\pi$.

		\item If $\pi$ is signed and $B$ is the boundary of a two-dimensional Delzant polytope (i.e., $X$ can be chosen as a four-dimensional toric symplectic manifold), then any realization $\PP(E)$ as above admits both a K\"ahler structure and a $T^2$-invariant symplectic structure.
		\item If, in the situation of (iii), $\pi$ is of twisted type and $\Gamma$ has maximal, i.e., $n-1$ interior vertices, where $B$ has is an $n$-gon and $n\neq 4$, then the realization does not admit an invariant K\"ahler structure.
	\end{enumerate}
\end{thm}

Here, a vertex $v$ of a $3$-valent signed $T^2$-GKM graph is called \emph{interior} if the cone spanned by the weights of the three edges leaving $v$ span all of $\RR^2$.
The example above of a projectivization of a complex rank $2$ vector bundle over a blow-up of $\CC P^1 \times \CC P^1$ satisfies the assumptions of (iv): the fibration is of twisted type and there are exactly $4$ interior vertices. By the theorem above it is realized geometrically by a GKM manifold displaying the same exotic behaviour as the Examples of Tolman, Woodward and Eschenburg. In fact, by systematically studying graph fibrations on the combinatoric side, this method yields infinite families of such exotic Hamiltonian actions with arbitrarily large (finite) fixed point sets.

\subsection{The smooth GKM correspondence in dimension 6}

In our most recent paper \cite{GKZcorrespondence}, we investigated the general GKM correspondence in dimension $6$. Concerning GKM realization, we showed that the necessary conditions for realizability of a $3$-valent GKM graph described in Proposition \ref{prop:necessaryforrealization} are in fact sufficient.  More precisely:
	\begin{thm}\label{thm:we gotem all in dimension 6}
		Let $(\Gamma,\alpha)$ be an abstract $3$-valent $T^2$-GKM graph. Then the following statements hold true:
		\begin{enumerate}
			\item $(\Gamma,\alpha)$ is realizable by a $6$-dimensional simply-connected rational GKM manifold if and only if the graph cohomology $H^*(\Gamma;\alpha,\mathbb{Q})$ satisfies Poincaré duality.
			\item $(\Gamma,\alpha)$ is realizable by a simply-connected integer GKM manifold if and only if additionally the equivariant graph cohomology $H^*_{T^2}(\Gamma,\alpha;\mathbb{Z})$ is a free module over $H^*(BT^2;\mathbb{Z})$.
		\end{enumerate}
	\end{thm}
We also investigated the GKM rigidity question for simply-connected $6$-dimensional integer GKM manifolds: it turned out that in case all stabilizers are connected, the GKM graph determines the equivariant homeomorphism type of the manifold, while there are counterexamples in presence of disconnected isotropies. For integer GKM manifolds with connected stabilizers the situation is thus summarized in the following theorem:
\begin{thm}\label{thm:onetotonecorrespondence}
	We have a one-to-one correspondence
	\begin{center}
		\begin{tikzpicture}
			\node (a) at (0,0) {$\left\{\begin{array}{c}
						\text{6-dim.\ simply-connected}    \\
						\text{integer $T^2$-GKM manifolds} \\
						\text{with connected stabilizers}\end{array}\right\}$};

			\node (b) at (8.7,0) {$\left\{\begin{array}{c}
						\text{3-valent GKM graphs with Poincaré}       \\
						\text{duality, free (equivariant) cohomology,} \\
						\text{s.t.\ adjacent weights form a basis}\end{array}\right\}$};

			\draw[<->] (a) -- (b);
		\end{tikzpicture}
	\end{center}
	where the left hand side is considered up to equivariant homeomorphisms, and the right hand side up to isomorphisms of GKM graphs.
\end{thm}
We emphasize that, from the point of view of the realization question, the results of \cite{GKZcorrespondence} are far more general than those of \cite{GKZ} explained in the previous section. One can realize all $3$-valent GKM graphs that satisfy the necessary conditions as opposed to only the total spaces of fibrations. However, the realization procedure of \cite{GKZ} via fibrations has the advantage of translating combinatoric properties into the appropriate geometric structures on the realization. It remains an open question which abstract GKM graphs can be realized as almost complex, symplectic, or even Kähler GKM manifolds.

The idea of the realization theorem is to construct a GKM manifold out of a $3$-valent GKM graph by equivariant handle attachments. The GKM graph determines the one-skeleton equivariantly, and we may construct canonically an equivariant thickening of it, similar to \cite[Section 3.1]{GuilleminZaraII}. A connection on the GKM graph determines so-called connection paths which are determined by any two adjacent edges, by successively sliding them forward along one another using the connection. We attach equivariant $2$-handles to the boundary of the equivariant thickening along curves corresponding to the connection paths.
Gluing one handle for every connection path, we obtain a $T^2$-manifold with free action on the boundary, whose orbit space deformation retracts to the surface $F$ obtained by attaching discs to any connection path in the GKM graph. After some further handle attachments and modifications we can achieve that this orbit space is $S^2\times D^2$, where  $F$ is $S^2\times \{0\}$, and we obtain a closed manifold by capping this off with $S^3\times S^1\times T^2$.

Regarding the proof of the rigidity part of the statement the idea is that one knows the orbit space of such an action to be $S^4$ by a result of \cite{AyzenbergMasuda}. In the case of connected stabilizers the orbit space of the nonregular orbits is just the GKM graph which is embedded in some fashion in the orbit space $S^4$. As there are no knotted circles in $S^4$ any isomorphism between embedded GKM graphs extends to an ambient homeomorphism of the surrounding copies of $S^4$. We lift this argument to the equivariant world by first constructing an equivariant homeomorphism of the aforementioned thickening of the graph and then extending it over the free complements. The transfer to the equivariant world can also be achieved using the complexity $1$ rigidity results of Ayzenberg mentioned in Remark \ref{rem:rigidity}.

\begin{ex}\label{ex:rigidity kaputt in dim 6}
	We sketch an example of equivariant non-rigidity in dimension $6$ in presence of nontrivial discrete isotropies, taken from \cite[Section 5]{GKZcorrespondence}. Consider the $T:=T^2$-action on $M:=S^2\times S^2\times S^2$ given by the pullback of the standard componentwise $T^3$-action via $(s,t)\mapsto (s,st^2,st^5)$. The orbit space of the $T^3$-action is a cube, the one-skeleton of which is the GKM graph of both the $T^2$- and the $T^3$-action. Considering the $T$-graph, at each vertex three edges with the labels $(1,0)$, $(1,2)$ and $(1,5)$ meet, which have the property that they span $\ZZ^2$, but any two of them do not form an integer basis of $\ZZ^2$.  The orbit space of the $T^2$-action is $S^4$; considering
	\[
		N:=\{p\in M\mid T_p\neq \{e\}\}\subset M,
	\]
	the set of points with nontrivial isotropy, we have that its orbit space $N/T$ is given by the boundary of the cube, i.e., it is homeomorphic to $S^2\subset S^4$. A space equivariantly homeomorphic to $M$ does not only need to have the same orbit space, but also the way in which the set $S^2$ of singular orbits is embedded in $S^4$ has to be the same as the embedding $N/T\subset M/T$.

	It was shown by Artin \cite{Artin} that there exist knotted two-spheres in $S^4$: he considered arcs in the real half space $\{(x_1,x_2,x_3)\mid x_3\geq 0\}$ starting and ending in the hyperplane $x_3=0$ and rotated it around the plane $x_3=x_4=0$ to obtain a sphere in $\RR^4$, hence by one-point compactification in $S^4$. For any such arc, the complement of the resulting sphere in $S^4$ has the same fundamental group as the complement of the knot in $S^3$ obtained by joining the end points of the initial arc. Hence starting with a non-trivial knot we obtain a non-trivially knotted $2$-sphere.

	Returning to the original construction we note that a small closed equivariant neighborhood $Y$ of $N$ in $M$ has orbit space $S^2\times D^2$, and satisfies $\partial Y = S^2\times S^1\times T$. Now let $Z$ be the complement of a small tube $S^2\times D^2$ around a knotted $S^2\subset S^4$. We replace in our space above the complement of $Y$ in $M$ by $Z\times T$ and obtain a new simply-connected integer GKM manifold. In other words, we realize $N/T$ as a knotted sphere in the orbit space $S^4$ of a $T$-action. For different choices of $Z$ we obtain examples of GKM manifolds that are not equivariantly homeomorphic, although they all share the same GKM graph.

	The idea behind this example is that knotting phenomena are typically not detected by cohomology, in the same way as all knot groups, i.e., fundamental groups of complements of knots in $S^3$ have the same abelianization, namely $\ZZ$. Thus the GKM graph of an action, which (at least for rational coefficients) contains the same information as the equivariant cohomology of the action \cite{FranzYamanaka} (see also \cite{GZreconstruct}), is too weak an invariant to distinguish such $T$-spaces.
\end{ex}

\bibliographystyle{acm}
\bibliography{GKMsurvey}

\end{document}